\documentclass[12pt]{amsart}
\usepackage[utf8]{inputenc}

\theoremstyle{plain}
\usepackage[mathscr]{eucal}
\usepackage{setspace}
\usepackage{amsmath, amsthm}
\usepackage{amssymb}
\usepackage{mathrsfs}
\usepackage{bbding}
\usepackage{bbold}
\usepackage{graphicx}
\usepackage{pifont}
\usepackage{verbatim}
\usepackage{eucal}
\usepackage{fancyhdr}
\usepackage{bm}
\usepackage{upgreek}
\usepackage{mathtools}
\usepackage{stmaryrd}
\usepackage [english]{babel}
\usepackage [autostyle, english = american]{csquotes}
\MakeOuterQuote{"}
\usepackage{hyperref}

\usepackage{dutchcal}
\usepackage{xcolor}

\newtheorem{theorem}{Theorem}[section]
\newtheorem{lemma}[theorem]{Lemma}

\newtheorem{claim}{Claim}
\newtheorem{corollary}[theorem]{Corollary}
\newtheorem{proposition}[theorem]{Proposition}

\theoremstyle{definition}
\newtheorem{definition}[theorem]{Definition}

\theoremstyle{remark}

\advance \topmargin by -\headheight
\advance \topmargin by -\headsep
\evensidemargin \oddsidemargin
\numberwithin{equation}{section}
\DeclareSymbolFont{AMSb}{U}{msb}{m}{n}
\DeclareMathSymbol{\N}{\mathbin}{AMSb}{"4E}
\DeclareMathSymbol{\Z}{\mathbin}{AMSb}{"5A}
\DeclareMathSymbol{\R}{\mathbin}{AMSb}{"52}
\DeclareMathSymbol{\Q}{\mathbin}{AMSb}{"51}
\DeclareMathSymbol{\C}{\mathbin}{AMSb}{"43}
\renewcommand{\emptyset}{\varnothing}
\renewcommand{\ll }{\langle\hspace{-.7mm}\langle }
\newcommand{\rr }{\rangle\hspace{-.7mm}\rangle }
\DeclareMathAlphabet\mathbfcal{OMS}{cmsy}{b}{n}

\title{A continuum of non-measure equivalent groups}

\begin{document}

\begin{abstract}
We construct a continuum sized family $\{G_x\}_{x\in\{0,1\}^{\mathbb N}}$ of pairwise non-measure equivalent countable groups which have property (T) (hence are finitely generated), have zero $\ell^2$-Betti numbers of all orders, and are torsion-free.

We also prove that the equivalence relation $\simeq^{\mathrm{fg}}_{\mathrm{ME}}$ of measure equivalence between finitely generated groups is non-smooth, resolving a question of S. Thomas. 
Our proof moreover shows that $\simeq^{\mathrm{fg}}_{\mathrm{ME}}$ sits above every countable Borel equivalence relation in the Borel reducibility hierarchy.
\end{abstract}

\author{Adrian Ioana}
\address{Department of Mathematics \\ University of California San Diego \\ 9500 Gilman Drive \\ La Jolla \\ CA
92093 \\ USA}
\email{aioana@ucsd.edu}
\thanks{AI was supported by NSF grants DMS-2153805 and DMS-2451697, and the Presidential Chair in Mathematics}

\author{Robin Tucker-Drob}
\address{Department of Mathematics \\ University of Florida \\ 1400 Stadium Road \\ Gainesville \\ FL 32611 \\ USA}
\email{r.tuckerdrob@ufl.edu}
\thanks{RTD was supported by NSF grant DMS-2246684}

\maketitle
\section{Introduction}
The notion of measure equivalence was introduced by Gromov in \cite{Gr}  as a measurable counterpart to the notion of quasi-isometric equivalence of groups (see the survey \cite{Fu}). Two countably infinite groups $G$ and $H$ are said to be {\it measure equivalent} if they admit commuting measure preserving actions on an infinite measure space, such that both of their actions  have fundamental domains of finite measure. Several natural families of countable groups are known to be pairwise measure equivalent: (i) lattices in a fixed locally compact second countable group, (ii) infinite amenable groups \cite{OW}, and (iii) finitely generated non-abelian free groups. 

On the other hand, in order to distinguish groups up to measure equivalence, one uses either measure equivalence invariants or  rigidity results. 
Measure equivalence invariants come in two flavors. First, we have binary invariants associated with properties that are preserved under measure equivalence. Examples of such properties include amenability, Kazhdan's property (T), Haagerup's property, containment in various classes ($\mathcal C_{\text{reg}}$, $\mathcal C$ \cite{MS06} and $\mathcal S$ \cite{Sa}), treability \cite{KM}, weak amenability \cite{Jo} and proper proximality \cite{BIP,IPR}. Second, we have numerical invariants that are preserved (or scaled appropriately) under measure equivalence, including the Cowling-Haagerup invariant \cite{CH,CZ,Jo}, cost (for groups of fixed price) \cite{Ga00}, $\ell^2$-Betti numbers and ergodic dimension \cite{Ga02}.
 
 Measure equivalence rigidity results aim to show that, within certain families of groups, various aspects of the groups can be recovered from their measure equivalence classes, see, e.g., \cite{MS06,Sa,AG,HH22,EH}.
 The most extreme rigidity results in this context completely determine the measure equivalence class of a given group, often showing that it remembers the group itself. Such superrigidity phenomena have been obtained for several groups, including higher-rank lattices \cite{Fu99}, mapping class groups \cite{Ki10}, surface braid groups \cite{CK}, Out$(\mathbb F_N)$ for $N\geq 3$ \cite{GH},  and generalized Higman groups \cite{HH25}. 
 
 The works discussed above, whether they give measure equivalence invariants or establish rigidity statements, lead to infinite but typically countable families of pairwise non-measure equivalent groups. As such, it remains a challenging problem to construct large (uncountable) families of non-measure equivalent groups. We contribute to this problem here by proving the following:
 
\begin{theorem}\label{A}
There exists a continuum sized family $\{G_x\}_{x\in \{0,1\}^\mathbb N}$ of pairwise non-measure equivalent  ICC countable groups
 which have property (T) (and hence are finitely generated).

Moreover, each $G_x$ satisfies $\beta^{(2)}_k(G_x)=0$, for every $k\geq 1$, and can be taken torsion-free.
\end{theorem}

Our construction is flexible enough to produce the following strengthening of Theorem \ref{A}, using the framework of Borel reducibility; see \S\ref{sec:Borel} for the relevant definitions.

\begin{theorem}\label{ctblered} Let $\simeq^{\mathrm{fg}}_{\mathrm{ME}}$ be the equivalence relation of measure equivalence on the space $\mathcal{G}_{\mathrm{fg}}$ of marked finitely generated groups.
There exists a Borel subset $Z$ of $\mathcal{G}_{\mathrm{fg}}$, consisting of ICC, torsion-free, property (T) groups whose $\ell^2$-Betti numbers all vanish, such that 
\begin{enumerate}
\item Every countable Borel equivalence relation is Borel reducible to $\simeq^{\mathrm{fg}}_{\mathrm{ME}}|_Z$.  

\item The isomorphism relation on $Z$ is a universal countable Borel equivalence relation.
\end{enumerate}
Moreover, there is a property (T) group $V$ such that every group in $Z$ is a quotient of $V$. 
\end{theorem}

Part (1) implies that measure equivalence on the space of marked finitely generated groups is non-smooth, thereby answering a longstanding question of Thomas \cite{ThomasCayleygraphs}. 
Part (2) proves a conjecture of Thomas and Williams. In  \cite[Theorem 1.5]{TW}, they prove that isomorphism of property (T) groups is weakly universal, and in \cite[Conjecture 1.6]{TW} they conjecture that it is, in fact, universal.

Gaboriau's work on cost and $\ell^2$-Betti numbers \cite{Ga00,Ga02} can be used to construct continuum many  countable groups which are pairwise non-measure equivalent, see Proposition \ref{prop:direct_sum_family}; note, however, that these groups  are  neither property (T) 
nor torsion-free. In \cite{LN}, L\'{o}pez Neumann constructed an infinite, but countable, family of finitely presented simple property (T) groups which are pairwise non-measure equivalent. 
Recently, Drimbe and Vaes found in \cite[Theorem A]{DV} a continuum of torsion-free  but not finitely generated countable groups which are pairwise not von Neumann equivalent and in particular not measure equivalent.

In a preliminary version of our paper, we proved the main assertion of Theorem \ref{A} for a family of groups $\{G_x\}_{x\in\{0,1\}^{\mathbb N}}$ 
which were constructed using Proposition \ref{prop:direct_sum_family} and thus  contain torsion. Independently of this preliminary result, Fournier-Facio and Sun found in \cite[Theorem D]{FS} a continuum $(M_S)_S$ of non-measure equivalent finitely generated groups which are also  torsion-free  and which can also be chosen to have property (T). 
The groups  constructed in \cite{FS} are  indexed by sets $S\subset\mathbb N_{\geq 3}$ and have the remarkable property that the set $S$ can be
recovered from the $\ell^2$-Betti numbers $\{\beta_n^{(2)}(M_S)\}_{n\geq 3}$, considered up to a  scalar multiple. Since measure equivalent groups have proportional $\ell^2$-Betti numbers by \cite{Ga02},  the groups $(M_S)$ are non-measure equivalent.

After \cite{FS} was posted on the arXiv, we showed that a variation of our construction (using Proposition \ref{torsion-free-family} instead of Proposition \ref{prop:direct_sum_family}) ensures that the groups $\{G_x\}_{x\in\{0,1\}^{\mathbb N}}$ can be taken torsion-free, thus yielding the moreover assertion of Theorem \ref{A}. Since the  groups $\{G_x\}_{x\in\{0,1\}^{\mathbb N}}$ have zero $\ell^2$-Betti numbers of all orders, they are distinct from the groups considered in \cite{FS}.
Indeed, a principal novelty of Theorem \ref{A} consists of providing a continuum of non-measure equivalent finitely generated groups in the regime where the usual invariants -- $\ell^2$-Betti numbers and cost -- are trivial and thus provide no useful information.

In the remainder of the introduction, we outline the proof of Theorem \ref{A}. The groups $\{G_x\}_{x\in\{0,1\}^{\mathbb N}}$ arise as {\it wreath-like} products in the sense introduced by Chifan, Osin, Sun and the first author in  \cite{CIOS23a,CIOS23b}. We refer the reader to Subsection \ref{Wreath} for the definition of the class $\mathcal W\mathcal R(A,B)$ of wreath-like products of groups $A$ and $B$.  
Here we only note that the classical wreath product $A\wr B$ belongs to $\mathcal W\mathcal R(A,B)$, and that there exist many wreath-like products with property (T) belonging to $\mathcal W\mathcal R(A,B)$, for certain hyperbolic groups $B$.

The construction of $\{G_x\}_{x\in\{0,1\}^{\mathbb N}}$ relies on the following four ingredients, which appear to be of independent interest:
\begin{enumerate}
\item (Proposition \ref{torsion-free-family}) There exists an infinite family $\{A_k\}_{k\geq 1}$ of pairwise non-measure equivalent, torsion-free, ICC, $4$-generated groups belonging to the class $\mathcal C$ of \cite{MS06}. 
\vskip 0.1in
\item (Proposition  \ref{WR}) For every $n\in\mathbb N$, there exists an ICC nonamenable normal subgroup $D$ of a torsion-free hyperbolic group, with the following property: for any sequence $(C_k)_{k\in\mathbb N}$ of $n$-generated groups, the class  $\mathcal W\mathcal R(\bigoplus_{k\in\mathbb N}C_k,D)$ 
 contains a property (T) group. 
\vskip 0.1in

\item (Theorem \ref{descend}) Let  $A,A'$ be countable groups and $B,B'$  be ICC nonamenable subgroups of hyperbolic groups. If $G\in\mathcal W\mathcal R(A,B)$ and $G'\in\mathcal W\mathcal R(A',B')$ are measure equivalent, then $A^{(B)}$ and $A'^{(B')}$ are measure equivalent.
\vskip 0.1in

\item (Theorem \ref{thm:MSinfiniteICCsum}) If $(C_k)_{k\in\mathbb N},(C'_k)_{k\in\mathbb N}$ are ICC groups in the class $\mathcal C$ such that the product groups $\bigoplus_{k\in\mathbb N}C_k, \bigoplus_{k\in\mathbb N}C_{k}'$ are measure equivalent, then modulo a permutation of $\mathbb N$, the groups $C_k,C_k'$ are measure equivalent, for every $k\in\mathbb N$.
\end{enumerate} 

Theorem \ref{thm:MSinfiniteICCsum} extends  Monod and Shalom's well-known measure equivalence rigidity for products theorem \cite[Theorem 1.16]{MS06} from finite to infinite products, and from torsion-free to ICC groups in the class $\mathcal C$.
The same conclusion with class $\mathcal C$ groups replaced by nonamenable biexact groups was established recently by Ding and Drimbe in \cite[Theorem E]{DD}. 

The article \cite{DD} also established (as part of the proof of \cite[Corollary F]{DD}) the following case of Theorem \ref{descend}: if $A,A'$ are countable groups, $B,B'$  are ICC hyperbolic groups, and the wreath product groups $A\wr B,A'\wr B'$ are measure equivalent, then $A^{(\mathbb N)}$ and $A'^{(\mathbb N)}$ are measure equivalent. 
Note that in the proof of Theorem \ref{A} we use Theorem \ref{descend} in the case when $G\in\mathcal W\mathcal R(A,B),G'\in\mathcal W\mathcal R(A',B')$ have property (T), and thus are not wreath products. For other measure equivalence rigidity results for wreath product groups, see \cite{Sako09,CPS}; see \cite{TDW24} for some measure equivalence anti-rigidity results for wreath products.

Finally, we explain briefly how ingredients (1)-(4) are combined to prove Theorem \ref{A}. Let $\{A_k\}_{k\geq 1}$ be the groups provided by (1). 
For $x\in\{0,1\}^{\mathbb N}$ (viewed as a subset of $\mathbb N$), let $D_x=\bigoplus_{k\in x}A_k$. Then (2) implies the existence of a property (T) group $G_x\in\mathcal W\mathcal R(D_x,B)$. Moreover, $G_x$ is torsion-free since it is built from the torsion-free groups $\{A_k\}_{k\geq 1}$ and $B$.  
If $G_x$ and $G_y$ are measure equivalent, then (3) implies that $D_x^{(\mathbb N)}\cong\oplus_{k\in x}A_k^{(\mathbb N)}$ and $D_y^{(\mathbb N)}\cong\oplus_{k\in y}A_k^{(\mathbb N)}$ are also measure equivalent. Using the defining properties of the sequence $\{A_k\}_{k\geq 1}$ from (1), the product rigidity statement from (4) implies that $x=y$. 

\subsection*{Acknowledgements and AI tool disclosure} We are grateful to Changying Ding for helpful comments.

Gemini was used for English language proofreading and grammatical corrections. All mathematical content was generated solely by the human authors.

\section{Preliminaries}

\subsection{Terminology}
We start this section with some notation which we will use later on.

For a standard Borel space $Y$, we denote by $\mathcal P(Y)$ the space of Borel probability measures  on $Y$, and by $\mathcal P_{\leq k}(Y)$ (respectively, $\mathcal P_{\geq k}(Y)$) the space  of Borel probability measures  on $Y$ whose support has at most (respectively, at least) $k$ points. If $Y$ is countable,  we identify $\mathcal P(Y)$ with $\ell^1(Y)_{+,1}:=\{f\in\ell^1(Y)\mid  f\geq 0,\|f\|_1=1\}$.  For $\eta\in\mathcal P(Y)$, we denote by $\text{supp}(\eta)$  and $\|\eta\|$ its support and norm.
We denote by $\mathcal F(Y)$ the collection of nonempty finite subsets of $Y$.

Let $Y$ be a standard Borel space. An equivalence relation $\mathcal R\subset Y\times Y$ is called a {\it countable Borel equivalence relation} if $\mathcal R$ is a Borel subset of $Y\times Y$ and 
the $\mathcal R$-equivalence class of every $y\in Y$ is countable.
 If $G\curvearrowright Y$ is a Borel action of a countable group $G$, then the {\it orbit equivalence relation} $\mathcal R(G\curvearrowright Y)=\{(y_1,y_2)\in Y\times Y\mid G\cdot y_1=G\cdot y_2\}$ is a countable Borel equivalence relation.

Let $G,H$ be countable groups. Let $(X,\mu)$ be a probability space,  which we always assume to be standard,  i.e.,  $X$ is a standard Borel space and $\mu\in\mathcal P(X)$.
 Let $G\curvearrowright (X,\mu)$ be a probability measure preserving ({\it pmp}) action. A Borel map $c:G\times X\rightarrow H$ is called a {\it cocycle} if it satisfies $c(g_1g_2,x)=c(g_1,g_2x)c(g_2,x)$, for every $g_1,g_2\in G$ and a.e. $x\in X$. 
 We say that $c$ is {\it cohomologous} to a cocycle $c':G\times X\rightarrow H$ if there exists a Borel map $\varphi:X\rightarrow H$ such that $c'(g,x)=\varphi(gx)c(g,x)\varphi(x)^{-1}$, for every $g\in G$ and a.e. $x\in X$. 

\subsection{Measure equivalence}
Two countable groups $G, H$ are  {\it measure equivalent (ME)} \cite{Gr} if there exists a standard infinite measure space $(\Omega,m)$ and a measure preserving action $G\times H\curvearrowright (\Omega,m)$ such that both the $G$- and $H$-actions admit a finite measure (measurable) fundamental domain. Any such space $(\Omega,m)$ is called an {\it ME-coupling} of $G$ and $H$.

Let $Y,X\subset\Omega$ be fundamental domains for the actions of $G$ and $H$, respectively. Define $\nu=m(Y)^{-1}m_{|Y}\in\mathcal P(Y)$ and $\mu=m(X)^{-1}m_{|X}\in\mathcal P(X)$.
Then we have pmp actions $G\curvearrowright (X,\mu)$ and $H\curvearrowright (Y,\nu)$ given by $\{g\cdot x\}=gH x\cap X$ and $\{h\cdot y\}=hG y\cap Y$. Moreover, we have cocycles $\alpha:G\times X\rightarrow H$ and $\beta:H\times Y\rightarrow G$ given by $g\cdot x=g\alpha(g,x)x$ and $h\cdot y=h\beta(h,y)y$.  Cocycles arising this way are called {\it $ME$-cocycles} \cite{Fu}.

\subsection{Essentially trivial, amenable and elementary cocycles}

In the next section we will prove a rigidity result for cocycles with values into hyperbolic groups (Theorem \ref{CR}) and derive a corollary (Corollary \ref{CR2}) needed in the proof of  our main result. In preparation for these results, we introduce the following terminology here.

\begin{definition}\label{cocycles}
Let $G,H$ be  countable groups, $G\curvearrowright (X,\mu)$ a pmp action and $c:G\times X\rightarrow H$  a cocycle.

 \begin{enumerate}
\item  If $S$ is a set endowed with an action of $H$, a map $\varphi:X\rightarrow S$ is called {\it $c$-equivariant} if it satisfies that $\varphi(gx)=c(g,x)\varphi(x)$, for every $g\in G$ and a.e. $x\in  X$.
\item We say that $c$ is 
{\it essentially trivial} if there is a $c$-equivariant Borel map $\varphi:X\rightarrow\mathcal F(H)$.
\item If $H$ is a hyperbolic group and $\partial H$ denotes its Gromov boundary,  then we say that $c$ is  {\it elementary} if there is a $c$-equivariant Borel map $\varphi:X\rightarrow \mathcal P_{\leq 2}(\partial H)$.
\item   We say that  $c$ is  {\it amenable} if there exist Borel maps $\varphi_n:X\rightarrow\mathcal P(H)$, $n\in\mathbb N$, such that $\|\varphi_n(gx)-c(g,x)\varphi_n(x)\|_1\rightarrow 0$, for every $g\in G$ and a.e. $x\in X$.
\end{enumerate}

In (2)-(4), we consider the natural actions of $H$ on $\mathcal F(H), \mathcal P(H)$, and $\mathcal P_{\leq 2}(\partial H)$, respectively.

\end{definition}
 We note that the above  notion of an elementary cocycle is a variation of (but different from) the notion introduced in \cite[Definition 1.1]{MS}.  Thus, $c$ is elementary in the sense of \cite[Definition 1.1]{MS} if and only if it essentially trivial or elementary in the sense of our Definition \ref{cocycles}.
 The notion of an amenable cocycle is inspired by \cite{AD} and the proof of \cite[Theorem 4.1]{BDV}. If $H$ is amenable, then any cocycle $c:G\times X\rightarrow H$ is amenable.
 
 In Corollary \ref{amenelem} we clarify the relationship between the above notions if $H$ is hyperbolic. Thus, we show that any elementary cocycle is amenable, and, conversely, any amenable cocycle whose restriction to any $G$-invariant set is not essentially trivial must be elementary.

We next record a few basic facts concerning the above notions.
\begin{proposition}\label{ess_triv_amenable}  
Let $G,H$ be  countable groups, $G\curvearrowright (X,\mu)$ a pmp action and $c:G\times X\rightarrow H$  a cocycle.
Then the following hold:

\begin{enumerate} \item $c:G\times X\rightarrow H$ is essentially trivial if and only if there exists a $c$-equivariant Borel map $\varphi:X\rightarrow\mathcal P(H)$. 
In particular, if $c$ is essentially trivial then it is amenable.
\item If $\{G_k\}_{k\geq 1}$ is an increasing sequence of subgroups of $G$ with  $G=\cup_{k\geq 1}G_k$, then a cocycle $c:G\times X\rightarrow H$ is amenable if and only if $c_{|(G_k\times X)}$ is amenable for all $k\geq 1$.
\item If $G$ has property (T) and $c$ is amenable, then $c$ is essentially trivial.
\item Let $\{H_k\}_{k\geq 1}$ be an increasing sequence of subgroups of $H$ with $H=\cup_{k\geq 1}H_k$. Assume that   $H=H_k\times H_k'$ for some subgroup $H_k'$ of $H$, and let $p_k:H\rightarrow H_k$ be the quotient homomorphism, for every $k\geq 1$. If the cocycle $p_k\circ c:G\times X\rightarrow H_k$ is amenable,  for every $k\geq 1$, then $c$ is amenable.
\item Assume that $c$ is an ME-cocycle arising from an ME-coupling of $G$ and $H$. If $c_{|(G_0\times X)}$ is an amenable cocycle, for some subgroup $G_0<G$, then $G_0$ is amenable.
\item Assume that $c$ is an ME-cocycle arising from an ME-coupling of $G$ and $H$. Let $N\lhd H$ be a normal subgroup and denote by $p:H\rightarrow H/N$ the quotient homomorphism. If $p\circ c:G\times X\rightarrow H/N$ is an amenable cocycle, then the group $H/N$ is amenable.
\item Let $L$ be a countable group with $H\leq L$. 
Then $c$ is essentially trivial as an $H$-valued cocycle if and only if $c$ is essentially trivial as an $L$-valued cocycle.
Likewise, $c$ is amenable as an $H$-valued cocycle if and only if $c$ is amenable as an $L$-valued cocycle.
\end{enumerate}

\end{proposition}

\begin{proof} The proofs of parts (1)-(3) are standard and are left to the reader.

(4) Assume that $p_k\circ c$ is amenable for every $k\geq 1$. For $k\geq 1$, view $\mathcal P(H_k)\subset\mathcal P(H)$ using the embedding $H_k<H$. 
Since $p_k\circ c$ is amenable, we can find a sequence of Borel maps $\xi_{n,k}:X\rightarrow\mathcal P(H)$ such that $\|\xi_{n,k}(g x)-p_k(c(g,x))\cdot\xi_{n,k}(x)\|_1\rightarrow 0$, as $n\rightarrow\infty$, for every $g\in G$ and a.e. $x\in X$. Thus, we can find $n_k\geq 1$ such that $\eta_k:=\xi_{n_k,k}:X\rightarrow \mathcal P(H)$ satisfies $\|\eta_k(g x)-p_k(c(g,x))\eta_k(x)\|_1\rightarrow 0$, as $k\rightarrow\infty$, for every $g\in G$ and a.e. $x\in X$. Since for every $h\in H$ we have that $p_k(h)=h$, for large enough $k\geq 1$, we conclude that $\|\eta_k(gx)-c(g,x)\eta_k(x)\|_1\rightarrow 0$, for every $g\in G$ and a.e. $x\in X$. Hence, $c$ is amenable.

(5)  Assume that $c$ is an ME-cocycle. Thus, there exists an ME-coupling $(\Omega,m)$ of $G$ and $H$ such that we can view $X\subset\Omega$ in such a way that $X$ is a fundamental domain for the action of $H$ and $c$ is given by $g\cdot x=gc(g,x)x$, for every $g\in G$ and $x\in X$. Assuming that $c|(G_0\times X)$ is an amenable cocycle, we can find a sequence of Borel maps $\xi_n:X\rightarrow \mathcal P(H)$ such that $\|\xi_n(g\cdot x)-c(g,x)\xi_n(x)\|_1\rightarrow 0$, for every $g\in G_0$ and a.e. $x\in X$. For every $n\geq 1$, we define a Borel function $\eta_n\in\text{L}^1(\Omega,m)_{+,1}$ by letting $\eta_n(x)=\xi_n(hx)(h)$, for every $x\in \Omega$ and $h\in H$ such that $hx\in X$. A direct calculation shows that if $g\in G$, then $$\int_\Omega|\eta_n(gx)-\eta_n(x)|\;\text{d}m(x)=\int_X\|\xi_n(g\cdot x)-c(g,x)\xi_n(x)\|_1\;\text{d}m(x).$$
From this we deduce that $\|g\cdot\eta_n-\eta_n\|_{\text{L}^1(\Omega,m)}\rightarrow 0$, for every $g\in G_0$. Since the action of $G$ on $\Omega$ has a measurable fundamental domain, the action of $G_0$ also has a measurable fundamental domain, say $Z$. If we define $\zeta_n\in\ell^1(G_0)_{+,1}=\mathcal P(G_0)$ by letting 
$\zeta_n(g_0)=\int_{g_0Z}\eta_n\;\text{d}m$, then it follows that $\|g\cdot\zeta_n-\zeta_n\|_1\rightarrow 0$, for every $g\in G_0$. This proves that $G_0$ is amenable.

(6) Assume that $c$ is an ME-cocycle  and keep the notation from the proof of (5). Let $Y\subset\Omega$ be a fundamental domain for the action of $G$ and consider the associated action of $H$ on $(Y,\nu)$, where $\nu=m(Y)^{-1}m_{|Y}$.  After possibly replacing $Y$ by $gY$, for some $g\in G$, we may assume that $Z:=X\cap Y$ is non-negligible.
Assume that $p\circ c:G\times X\rightarrow H/N$ is amenable.
Thus, there exists a sequence of Borel maps $\xi_n:X\rightarrow\mathcal P(H/N)$ such that $\|\xi_n(g\cdot x)-p(c(g,x))\xi_n(x)\|_1\rightarrow 0$, for every $g\in G$ and a.e. $x\in X$.

 Let $z\in Z\subset Y$ and $h\in H$ such that $h\cdot z\in Z$. Then $h\cdot z=gh z$, for some $g\in G$.
Since $z,g h z\in Z\subset X$, we get that $h=c(g,z)$ and that $h\cdot z=g\cdot z$. From this we deduce that $\|\xi_n(h\cdot z)-p(h)\xi_n(z)\|_1=\|\xi_n(g\cdot z)-p(c(g,z))\xi_n(z)\|_1\rightarrow 0$, for all $h\in H$ and almost every $z\in Z$ such that $h\cdot z\in Z$. Enumerate $H=\{h_k\}_{k\geq 1}$, put $T= H\cdot  Z\subset Y$ and for every $n\geq 1$, define a Borel map $\eta_n:T\rightarrow \mathcal P(H/N)$ by letting $\eta_n(y)=p(h_k)\xi_n(h_k^{-1}\cdot y)$, where $k\geq 1$ is the smallest integer with $h_k\cdot y\in Z$. It is easy to see that $\|\eta_n(h\cdot y)-p(h)\cdot\eta_n(y)\|_1\rightarrow 0$, for every $h\in H$ and a.e. $y\in T$. 

But then the elements $\zeta_n\in\mathcal P(H/N)$ given by $\zeta_n=\nu(T)^{-1}\int_{T}\eta_n(y)\;\text{d}\nu(y)$
satisfy that $\|\zeta_n-p(h)\zeta_n\|_1\rightarrow 0$, for every $h\in H$. This proves that $H/N$ is amenable.

(7) It is clear that if $c$ is either essentially trivial or amenable as an $H$-valued cocycle, then $c$ also has the corresponding property when viewed as an $L$-valued cocycle.  

Let $T\subseteq L$ be a transversal for the right cosets of $H$ inside $L$ with $e\in T$, and let $p:L\to H$ be given by $p(ht)=h$ for $h\in H$, $t\in T$, so that $p$ is equivariant for the left translation actions of $H$.
If $\varphi : X\to\mathcal{F}(L)$ is map witnessing that $c$ is essentially trivial when viewed as an $L$-valued cocycle, then $H$-equivariance of $p$ (and the fact that $c$ takes values in $H$) ensures that the map $\varphi '(x)\coloneqq p(\varphi (x))$ is a witness to $c$ being essentially trivial as an $H$-valued cocycle.

The map $p$ induces an $H$-equivariant contractive linear map $p_* : \ell ^1(L)\to \ell ^1(H)$, given by $(p_*f)(h)\coloneqq \sum_{\ell \in p^{-1}(h)}f(\ell )$, taking $\mathcal{M}(L)$ to $\mathcal{M}(H)$.  
Thus, if the sequence $(\psi _n)_{n\in \N}$ witnesses amenability of $c$ as an $L$-valued cocycle, then the sequence $(\psi_n')_{n\in \N}$ of point-wise pushforwards $\psi_n'(x)\coloneqq p_*(\psi_n(x))$ witnesses amenability of $c$ as an $H$-valued cocycle.
\end{proof}

\subsection{Wreath-like product groups}\label{Wreath}  We continue by recalling from \cite{CIOS23a,CIOS23b} the notion of a wreath-like product of groups.

\begin{definition}
A group $G$ is called a \emph{wreath-like product} of two groups $A$ and $B$ if it is an extension of the form
$\{e\}\longrightarrow A^{(B)} \longrightarrow G \longrightarrow B\longrightarrow \{e\}$, 
where the action of $G$ on $A^{(B)}=\bigoplus_{b\in B}(A)_b$ by conjugation satisfies 
$g(A)_bg^{-1} = (A)_{\pi(g)b}$, for every $g\in G$ and $b\in B$, with $\pi:G\rightarrow B$ being the quotient homomorphism.  

\end{definition}
We denote by $\mathcal W\mathcal R(A,B)$ the family of all wreath-like products of $A$ and $B$.
The usual {\it wreath product} $A\wr B=A^{(B)}\rtimes B$ belongs to $\mathcal W\mathcal R(A,B)$. 
More precisely, $G\in\mathcal W(A,B)$ is (canonically) isomorphic to $A\wr B$ if and only if the extension used to define $G$ is split. 
The wreath-product $A\wr B$ does not have property (T) unless  $A$ has property (T) and $B$ is finite.
One of the main findings of \cite{CIOS23a,CIOS23b} is that there exists an abundance of wreath-like products with property (T). In particular, group theoretic Dehn filling was used to show that hyperbolic groups $G$ with property (T) admit many quotients which can be decomposed as wreath-like products and which inherit property (T) from $G$ (see \cite[Theorem 2.7]{CIOS23b}).

In the proof of Theorem \ref{A} we will need the following result.

\begin{proposition}\label{WR}
For every $n\in\mathbb N$,  there exists a 
property (T) group $V\in\mathcal W\mathcal R(\mathbb F_n^{(\mathbb N)},D)$, where $D$ is an ICC nonamenable normal subgroup of a torsion-free hyperbolic group.

Moreover, if $\{A_k\}_{k=1}^\infty$ is a sequence of $n$-generated groups, then there exists a property (T) group belonging to $\mathcal W\mathcal R(\bigoplus_{k=1}^\infty A_k,D)$.
\end{proposition}
This result is a variation of  \cite[Proposition 2.11]{CIOS23a} which established the case $n=1$.

\begin{proof}
As in the proof of \cite[Proposition 2.11]{CIOS23a}, let $S$ be any finitely presented, residually finite, torsion-free group with property (T). 
Fix also a nontrivial finitely presented torsion-free group $L$.
Applying \cite[Lemma 2.10]{CIOS23a}, there exists a short exact sequence $\{ e\}\to N\to G\stackrel{\gamma}\to S\times L\to \{ e\}$, where $G$ is torsion-free hyperbolic and $N$ is a non-trivial normal subgroup of $G$ with property (T).

Since $G_0:=\gamma^{-1}(S\times\{1\})$ is a normal subgroup of $G$ and $K(G)=\{1\}$ (as $G$ is torsion-free and $K(G)$ is the maximal finite normal subgroup of $G$), \cite[Lemma 3.22]{CIOS23b} implies that $G_0$ is a suitable subgroup of $G$, in the sense of \cite[Definition 3.21]{CIOS23b}. Applying \cite[Proposition 4.21]{CIOS23b} provides a Cohen-Lyndon subgroup $H\cong\mathbb F_n$ of $G$ such that $H$ is contained in $G_0$ and $B\coloneqq G/\ll H\rr$ is an ICC hyperbolic group, where  $\ll H\rr\lhd G$ denotes the normal subgroup generated by $H$. Since $G$ is torsion-free, we may also assume that  $B$ is torsion-free.

Since $G_0\lhd G$, we have $\ll H\rr\lhd G_0$.
Let $T=\{[sHs^{-1},tHt^{-1}\mid s\ll H\rr\not=t\ll H\rr\}$. By \cite[Corollary 4.18]{CIOS23b}, we get $W\coloneqq G/\langle T\rangle\in\mathcal W\mathcal R(H,G/\ll H\rr)$.

Let $\pi:G\rightarrow W$ be the quotient homomorphism, and define $V\coloneqq\pi(G_0)=G_0/T$ and $D\coloneqq G_0/\ll H\rr$. If $\rho:W\rightarrow B$ is the quotient homomorphism, then $V=\rho^{-1}(D)$. Thus, by \cite[Lemma 2.8]{CIOS23a} we get $V\in\mathcal W\mathcal R(H^{(I)},D)$, where $I$ is a set of cardinality equal to $[G:G_0]$.
Since $[G:G_0]=[S\times L:S\times\{1\}]=\infty$, we deduce that $H^{(I)}\cong \mathbb F_n^{(\mathbb N)}$.

Since $G_0$ sits in a short exact sequence $\{e\}\to N\to G_0\stackrel{\gamma}\to S\to \{e\},$ and $N$ and $S$ have property (T), we find that $G_0$ has property (T). Therefore, $V=\pi(G_0)$ has property (T). Since $D$ is a nontrivial normal subgroup of the torsion-free, ICC hyperbolic group $B$, and $H\cong\mathbb F_n$, it follows that $V$ is a group with the desired properties, proving the main assertion.

Since the groups $\{A_k\}_{k=1}^n$ are $n$-generated, their direct sum $\bigoplus_{k=1}^\infty A_k$ is a quotient of $\mathbb F_n^{(\mathbb N)}$. The moreover assertion then follows from \cite[Lemma 2.12]{CIOS23a}.
\end{proof}

\section{Cocycle rigidity}

The goal of this section is to prove the following theorem and its corollary needed in the proof of Theorem \ref{A}.

\begin{theorem}\label{CR}
Let $G_1,G_2,G$ be countable groups such that $G_1\times G_2\lhd G$ is a normal subgroup. Let $G\curvearrowright (X,\mu)$ be a pmp action and $c:G\times X\rightarrow H$ a cocycle, with $H$ a hyperbolic group. Then there is a measurable partition $X=X_1\sqcup X_2\sqcup Y$ such that $X_i$ is $(G_1\times G_2)$-invariant and $c|(G_i\times X)$ is essentially trivial, for $i\in \{1,2\}$, $Y$ is $G$-invariant and $c|(G\times Y)$ is elementary.\end{theorem}

The statement of Theorem \ref{CR} is inspired by Popa and Vaes's powerful dichotomy theorem \cite{PV} for von Neumann algebras arising from actions of hyperbolic groups. While Theorem \ref{CR} can be proved using \cite{PV}, we give a more elementary, ergodic-theoretic approach relying on results from \cite{Ad94a,Ad95,Ad96,HK,BDV}.


\begin{corollary}\label{CR2}
Let $G\in\mathcal W\mathcal R(A,B)$, for countable groups $A,B$. Let $G\curvearrowright (X,\mu)$ be an ergodic pmp action and $c:G\times X\rightarrow H$ be a cocycle, where $H$ is a subgroup of a hyperbolic group. Then $c$ is amenable or $c|(A^{(B)}\times X)$ is essentially trivial.
\end{corollary}

Corollary 
\ref{CR2} will be deduced at the end of the section from Theorem \ref{CR}.

In preparation for the proof of Theorem \ref{CR}, we first establish some additional basic facts concerning the notions in Definition \ref{cocycles}.

\begin{lemma}\label{normalizer}
Let $G<K,H$ be countable groups, $K\curvearrowright (X,\mu)$ a pmp action and $c:K\times X\rightarrow H$  a cocycle. Let $Y\subset X$ be a $G$-invariant measurable set. Define $Z=\cup_{k\in N_{K}(G)}kY$.
\begin{enumerate}
\item If $c|(G\times Y)$ is essentially trivial, then $c|(G\times Z)$ is essentially trivial.
\item If $H$ is hyperbolic and $c|(G\times Y)$ is elementary, then $c|(G\times Z)$ is elementary.
\item If $c|(G\times Y)$ is amenable, then $c|(G\times Z)$ is amenable.
\end{enumerate}
\end{lemma}

\begin{proof}
Let $S$ be a standard Borel space  endowed with a Borel action of $H$. Let $\varphi:Y\rightarrow S$ be a $c|(G\times Y)$-equivariant Borel map.
For $k\in K$, define $\varphi_k:k^{-1}Y\rightarrow S$ by $\varphi_k(x)=c(k,x)^{-1}\varphi(kx)$. 

\begin{claim}\label{equiv}
$\varphi_k$ is a $c|(k^{-1}Gk\times k^{-1}Y)$-equivariant Borel map. 
\end{claim}
Indeed, for every $g\in k^{-1}Gk$ and a.e. $x\in k^{-1}Y$, we have $kx\in Y$ and $kgk^{-1}\in G$, hence \begin{align*}
\varphi_k(gx)=c(k,gx)^{-1}\varphi(kgk^{-1}kx)= & c(k,gx)^{-1}c(kgk^{-1},kx)\varphi(kx)\\= & c(k,gx)^{-1}c(kgk^{-1},kx)c(k,x)\varphi_k(x)\\= & c(g,x)\varphi_k(x).
\end{align*}

 Claim \ref{equiv} gives that if $k\in N_K(G)$, then $\varphi_k$ is $c|(G\times k^{-1}Y)$-equivariant. Thus, if $c|(G\times Y)$ is essentially trivial (respectively, elementary), then  so is $c|(G\times k^{-1}Y)$, for every $k\in N_K(G)$. This implies parts (1) and (2). If $S=\mathcal P(H)$, then the same calculation as in the proof of Claim \ref{equiv} shows that $\|\varphi_k(gx)-c(g,x)\varphi_k(x)\|_1=\|\varphi(kgk^{-1}kx)-c(kgk^{-1},kx)\varphi(kx)\|_1$, for every $g\in k^{-1}Gk$ and $x\in k^{-1}Y$. This gives that  if $c|(G\times Y)$ is amenable and $k\in N_K(G)$, then $c|(G\times k^{-1}Y)$ is amenable, which implies part (3).
\end{proof}

Next, we clarify the relationship between elementary and amenable cocycles, see Corollary \ref{amenelem}. To this end, we first prove the following result.

\begin{lemma}\label{invmeasure}
Let $G,H$ be  countable groups, $G\curvearrowright (X,\mu)$ a pmp action and $c:G\times X\rightarrow H$ be an amenable cocycle. Let $H\curvearrowright S$ be an action by homeomorphisms on a compact metric space $S$. Then there exists a $c$-equivariant Borel map $\eta:X\rightarrow \mathcal P(S)$.\end{lemma}

The proof of this result is standard, but for lack of a reference, we include a detailed proof.

\begin{proof} Since $(X,\mu)$ is a standard probability space, we may assume that $X$ is a compact metric space.
Since $c$ is an amenable cocycle, we can find a sequence of Borel maps $\psi_n:X\rightarrow\mathcal P(H)$ such that $\|\psi_n(gx)-c(g,x)\psi_n(x)\|_1\rightarrow 0$, for every $g\in G$ and a.e. $x\in X$. Fix $s\in S$ and define $\pi:\mathcal P(H)\rightarrow \mathcal P(S)$ by letting $\pi(f)=\sum_{h\in H}f(h)\delta_{hs}$ for every $f\in\mathcal P(H)$. Then $\pi$ is a Borel $H$-equivariant map and $\|\pi(f)-\pi(g)\|=\|f-g\|_1$, for every $f,g\in\mathcal P(H)$. Thus, $\eta_n=\pi\circ\psi_n:X\rightarrow\mathcal P(S)$, $n\in\mathbb N$, are Borel maps such that $\|\eta_n(gx)-c(g,x)\eta_n(x)\|\rightarrow 0$, for every $g\in G$ and a.e. $x\in X$.

If  $n\in\mathbb N$ and $f:X\times S\rightarrow\mathbb C$ is a Borel function, then $ \int_Sf(x,t)\;\text{d}\eta_n(x)(t)=\sum_{h\in H}f(x,hs)\psi_n(x)(h)$, for every $x\in X$. This implies that $X\ni x\mapsto \int_Sf(x,t)\;\text{d}\eta_n(x)(t)\in\mathbb C$ is a Borel function, so we can define $\rho_n\in\mathcal P(X\times S)$ by $\rho_n:=\int_X(\delta_x\otimes\eta_n(x))\;\text{d}\mu(x)$.
Consider the Borel action $G\curvearrowright X\times S$ given by $g\cdot (x,t)=(gx,c(g,x)t)$. Then  $g\rho_n=\int_X(\delta_x\otimes c(g,g^{-1}x)\eta_n(g^{-1}x))\;\text{d}\mu(x)$. In combination with the dominated convergence theorem this implies that $$\text{$\|g\rho_n-\rho_n\|\leq\int_X \|c(g,g^{-1}x)\eta_n(g^{-1}x)-\eta_n(x)\|\;\text{d}\mu(x)\rightarrow 0$, for every $g\in G$.}$$
Therefore, any weak$^*$-limit point  $\rho\in\mathcal P(X\times S)$ of $\{\rho_n\}_{n\in\mathbb N}$ is $G$-invariant. Moreover, let $p:X\times S\rightarrow X$ be the projection $p(x,s)=x$. Then $p_*\rho_n=\mu$, for every $n\in\mathbb N$, and thus $p_*\rho=\mu$. Thus, we find a Borel map $\eta:X\rightarrow \mathcal P(S)$ such that $\rho=\int_X(\delta_x\otimes \eta(x))\;\text{d}\mu(x)$. Since  $\eta$ is unique up to measure zero and $g\rho=\int_X(\delta_x\otimes c(g,g^{-1}x)\eta(g^{-1}x))\;\text{d}\mu(x)$, we conclude that $\eta(x)=c(g,g^{-1}x)\eta(g^{-1}x)$, for all $g\in G$ and a.e. $x\in X$. This finishes the proof.
\end{proof}

An action of a countable group $H$ on a compact space $Y$ is called (topologically) {\it amenable} \cite{AD87} if there is a sequence of continuous maps $\eta_n:Y\rightarrow\mathcal P(H)$ such that for every $h\in H$ we have $\sup_{x\in Y}\|\eta_n(hx)-h\eta_n(x)\|_1\rightarrow 0$. 
Assume that $H\curvearrowright Y$ is an amenable action. Then the action $H\curvearrowright Y_0$ is  amenable for every closed $H$-invariant subset $Y_0\subset Y$.
 Moreover, the action $H\curvearrowright \mathcal P(Y)$ is also amenable,  where $\mathcal P(Y)$ is endowed with the weak$^*$-topology,
 as witnessed by the maps $\widetilde\eta_n:\mathcal P(Y)\ni \mu\mapsto \int_Y\eta_n(y)\;\text{d}\mu(y)\in\mathcal P(H)$. 
 
 By a result of Adams \cite[Theorem 5.1]{Ad94a} (see also \cite[Theorem 5.3.15]{BO}), for any hyperbolic group $H$,  the action $H\curvearrowright \partial H$ is amenable. Since $\mathcal P_{\leq 2}(\partial H)\subset\mathcal P(\partial H)$ is a closed $H$-invariant subset, we get that the actions $H\curvearrowright\mathcal P(\partial H)$ and  $H\curvearrowright \mathcal P_{\leq 2}(\partial H)$ are  also amenable for any hyperbolic group $H$.

\begin{corollary}\label{amenelem}
Let $G,H$ be countable groups, with $H$ hyperbolic, $G\curvearrowright (X,\mu)$ a pmp action and $c:G\times X\rightarrow H$ a cocycle. Then $c$ is amenable if and only if there is a $G$-invariant measurable set $Y\subset X$ such that $c|(G\times Y)$ is elementary and $c|(G\times (X\setminus Y))$ is essentially trivial.
\end{corollary}

\begin{proof}
Assume that $c$ is amenable. By Lemma \ref{invmeasure}, there is a $c$-equivariant Borel map $\varphi:X\rightarrow \mathcal P(\partial H)$. The conclusion follows from \cite[Lemma 3.2]{MS}; for completeness, we include the argument here.
Let $Y=\{x\in X\mid \varphi(x)\in\mathcal P_{\leq 2}(\partial H)\}$. Then $Y$ is a $G$-invariant measurable set and $c|(G\times Y)$ is elementary. Since $H$ is hyperbolic, there is an $H$-equivariant Borel map $\pi:\mathcal P_{\geq 3}(\partial H)\rightarrow\mathcal F(H)$ (see \cite[Corollary 5.3]{Ad96}). Since $\varphi(x)\in\mathcal P_{\geq 3}(\partial H)$, for every $x\in X\setminus Y$, the composition $\pi\circ\varphi$ witnesses that $c|(G\times (X\setminus Y))$ is essentially trivial. 

Since any essentially trivial cocycle is amenable (see Proposition \ref{ess_triv_amenable}(1)), in order  to prove the converse, it remains to argue that any elementary cocycle is amenable. Assume that $c$ is elementary and let $\varphi:X\rightarrow\mathcal P_{\leq 2}(\partial H)$ be a Borel $c$-equivariant map. 
Since the action $H\curvearrowright\mathcal P_{\leq 2}(\partial H)$ is amenable, there is a sequence of continuous maps $\eta_n:\mathcal P_{\leq 2}(\partial H)\rightarrow\mathcal P(H)$ such that $\|\eta_n(hy)-h\eta_n(y)\|_1\rightarrow 0$, for every $h\in H$ and  $y\in \mathcal P_{\leq 2}(\partial H)$.
The Borel maps $\psi_n:=\eta_n\circ\varphi:X\rightarrow\mathcal P(H)$ then satisfy $\|\psi_n(gx)-c(g,x)\psi_n(x)\|_1\rightarrow 0$, for every $g\in G$ and a.e. $x\in X$. In other words, $c$ is amenable.
\end{proof}

We next show that the elementarity of cocycles passes to normalizers. This relies on ideas of Adams (see \cite[Section 2]{Ad95}) and their refinement  by Hjorth and Kechris (see \cite[Theorem 2.2]{HK}).

\begin{proposition}\label{norm}
Let $G<K, H$ be countable groups, with $H$ hyperbolic, $K\curvearrowright (X,\mu)$ a pmp action and let $c:K\times X\rightarrow H$ be a cocycle. Assume that $c|(G\times X)$ is elementary and $c|(G\times Y)$ is not essentially trivial, for every non-negligible $G$-invariant measurable set $Y\subset X$. Then $c|(N_K(G)\times X)$ is elementary.
\end{proposition}

\begin{proof}
 Let $\mathcal S$ be the set of all $c$-equivariant Borel maps $\eta:X\rightarrow\mathcal P(\partial H)$. Since $c|(G\times X)$ is elementary,  $\mathcal S\not=\emptyset$. We claim that for every $\eta\in\mathcal S$, we have $\eta(x)\in\mathcal P_{\leq 2}(\partial H)$, for a.e. $x\in X$. 
Let $\eta\in\mathcal S$ and define $Y=\{x\in X\mid \eta(x)\in\mathcal P_{\geq 3}(\partial H)\}$. Then $Y$ is a $G$-invariant measurable set. Since $H$ is hyperbolic, there is an $H$-equivariant Borel map $\pi:\mathcal P_{\geq 3}(\partial H)\rightarrow\mathcal F(H)$ (see \cite[Corollary 5.3]{Ad96}). Then the Borel map $\varphi=:\pi\circ\eta:Y\rightarrow\mathcal F(H)$ is $c$-equivariant. Hence $c|(G\times Y)$ is essentially trivial and therefore $\mu(Y)=0$, which proves our claim.
 
By applying \cite[Proposition B4.1]{HK} (which goes back to work of Adams and Zimmer, see, e.g., \cite[Lemma 2.6]{Ad95}) we find $\eta\in\mathcal S$ such that for every $\eta'\in\mathcal S$ we have that $\text{supp}(\eta'(x))\subset\text{supp}(\eta(x))$, for a.e. $x\in X$. 
Moreover, we may assume that $\eta(x)$ assigns equal weights to the elements of its support, for every $x\in X$.
We claim that \begin{equation}\label{claim}\text{$\eta(kx)=c(k,x)\eta(x)$, for every $k\in N_K(G)$ and a.e. $x\in X$.}\end{equation}
To prove the claim, let
 $k\in N_K(G)$ and define $\eta_k:X\rightarrow\mathcal P_{\leq 2}(\partial H)$  by $\eta_k(x)=c(k,x)^{-1}\eta(kx)$. Claim \ref{equiv} from the proof of Lemma \ref{normalizer} shows that $\eta_k\in\mathcal S$. By the defining property of $\eta$, $\text{supp}(\eta_k(x))\subset\text{supp}(\eta(x))$, for a.e. $x\in X$. Putting $s(x)=|\text{supp}(\eta(x))|\in\{1,2\}$, we get $s(kx)\leq s(x)$, for a.e. $x\in X$. Since $\int_X(s(x)-s(kx))\;\text{d}\mu(x)=0$, we deduce that $s(kx)=s(x)$, for a.e. $x\in X$. Thus $|\text{supp}(\eta_k(x))|=|\text{supp}(\eta(x))|$, hence $\text{supp}(\eta_k(x))=\text{supp}(\eta(x))$ and $\eta_k(x)=\eta(x)$, for a.e. $x\in X$. This proves  \eqref{claim} which  implies the conclusion.
\end{proof}

Finally, to prove Theorem \ref{CR}, we will need a straightforward generalization of \cite[Theorem 4.1]{BDV}.  
Following \cite{BDV}, we say that a countable group $G$ has {\it property (S)} if there exists a map $\eta:G\rightarrow\mathcal P(G)$ such that $\lim_{x\rightarrow\infty}\|\eta(gxh)-g\eta(x)\|_1=0$, for every $g,h\in G$. We note that a group belongs to Ozawa's class $\mathcal S$ if and only if it is exact and has property (S), see \cite[Definition 15.1.2]{BO}.

\begin{theorem}[Brothier-Deprez-Vaes, \cite{BDV}]\label{BDV}
Let $G_1,G_2,H$ be countable groups, $G_1\times G_2\curvearrowright (X,\mu)$ a pmp action and  $c:(G_1\times G_2)\times X\rightarrow H$ a cocycle.
Assume that $H$ has property (S). 
Then there is a partition $X=X_1\sqcup X_2$ into measurable, $(G_1\times G_2)$-invariant sets such that
\begin{enumerate}
\item $c|(G_1\times X_1)$ is amenable.
\item $c|(G_2\times X_2)$ is essentially trivial.

\end{enumerate}
\end{theorem}

\begin{proof}
By using a maximality argument and Lemma \ref{normalizer},  in order to prove the conclusion, it suffices to show that if $Y\subset X$ is a non-negligible $(G_1\times G_2)$-invariant measurable set, then either $c|(G_1\times Y)$ is amenable or $c|(G_2\times Z)$ is essentially trivial, for a non-negligible $G_2$-invariant set $Z\subset Y$.

This assertion follows from the proof of \cite[Theorem 4.1]{BDV}. We sketch the argument, leaving the details to the reader.
Denoting $\omega=c|(G_1\times G_2\times Y)$, we have that either
\begin{enumerate}
\item there exists no sequence $g_n\in G_2$ such that $\omega(g_n,\cdot)\rightarrow\infty$ in measure, or
\item there exists a sequence $g_n\in G_2$ such that $\omega(g_n,\cdot)\rightarrow\infty$ in measure.
\end{enumerate}
In case (1), the proof of \cite[Theorem 4.1]{BDV} provides a Borel map $\xi:Y\rightarrow \ell^2(H)$ which is not zero a.e. and satisfies $\xi(gx)=\omega(g,x)\xi(x)$, for every $g\in G_2$ and a.e. $x\in Y$. Let $Z=\{x\in Y\mid \xi(x)\not=0\}$. For $x\in Z$, let $\varphi(x)=\{h\in H\mid \xi(x)(h)=\|\xi(x)\|_\infty\}$. Then $Z$ is non-negligible and $G_2$-invariant. Moreover,  $\varphi(x)\subset H$ is finite and $\varphi(g_2x)=\omega(g_2,x)\varphi(x)$, for all $g_2\in G_2$ and a.e. $x\in Z$, and thus $c|(G_2\times Z)=\omega|(G_2\times Z)$ is essentially trivial.

In case (2), the proof of \cite[Theorem 4.1]{BDV}  shows the existence of a sequence  of measurable functions $\psi_n:Y\rightarrow\mathcal P(H)$ such that the sequence of functions $\|\psi_n(gx)-\omega(g,x)\psi_n(x)\|_1$ converges to zero in measure, for every $g\in G_1$. After replacing $(\psi_n)$ with a subsequence we can assume that $\|\psi_n(gx)-\omega(g,x)\psi_n(x)\|_1\rightarrow 0$, for every $g\in G_1$ and a.e. $x\in Y$. This shows that $c|(G_1\times Y)=\omega|(G_1\times Y)$ is amenable, which finishes the proof.
\end{proof}

\begin{proof}[Proof of Theorem \ref{CR}]
Since every hyperbolic group $H$ has property $(S)$ (see \cite[Proposition 15.2.3 and Example 15.2.5]{BO}), Theorem \ref{BDV} gives a partition $X=Y_0\sqcup Y_2$ into measurable, $(G_1\times G_2)$-invariant sets such that
 $c|(G_1\times Y_0)$ is amenable and $c|(G_2\times Y_2)$ is essentially trivial. By Lemma \ref{normalizer}, there is a measurable $(G_1\times G_2)$-invariant subset $Y_1\subset Y_0$ such that $c|(G_1\times Y_1)$ is essentially trivial and $c|(G_1\times Z)$ is not essentially trivial, for any measurable $(G_1\times G_2)$-invariant non-null subset $Z$ of $V:=Y_0\setminus Y_1$.  
 
 Letting $Y=\cup_{g\in G}gV$, we will argue that
$c|(G\times Y)$ is elementary.
Since $c|(G_1\times V)$ is amenable, and  $c|(G_1\times Z)$ is not essentially trivial, for any measurable $(G_1\times G_2)$-invariant non-null subset $Z\subset V$,
Corollary \ref{amenelem} implies that $c|(G_1\times V)$ is elementary.
Further, Proposition \ref{norm} gives that $c|(G_1\times G_2\times V)$ is elementary. Since $G_1\times G_2$ is normal in $G$, Lemma \ref{normalizer} implies that $c|(G_1\times G_2\times Y)$ is elementary and $c|(G_1\times G_2\times W)$ is not essentially trivial, for any non-null measurable  $(G_1\times G_2)$-invariant set $W\subset Y$.
By applying Proposition \ref{norm} again we get that $c|(G\times Y)$ is amenable. 
If $W\subset Y$ is a non-null measurable $G$-invariant set, then $c|(G_1\times G_2\times W)$, and thus $c|(G\times W)$ is not essentially trivial.
Corollary \ref{amenelem} implies that $c|(G\times Y)$ is elementary.

Finally, since $X=Y_0\cup Y_2=V\cup Y_1\cup Y_2$ and $V\subset Y$,  we have that $X=Y\cup Y_1\cup Y_2$.
Letting $X_1=Y_1\setminus Y$ and $X_2=Y_2\setminus Y$, the conclusion follows.
\end{proof}

\begin{proof}[Proof of Corollary \ref{CR2}]
Recall that $G\in\mathcal W\mathcal R(A,B)$ means we have a short exact sequence $\{e\}\mapsto A^{(B)}\mapsto G\mapsto B\mapsto \{e\}$ such that if $\pi:G\rightarrow B$ is the quotient homomorphism, then $g(A)_bg^{-1}=(A)_{\varepsilon(g)b}$, for every $g\in G$ and $b\in B$.

By Proposition\ref{ess_triv_amenable}(7), it is enough to consider the case where $H$ itself is hyperbolic. 
Assuming that $c$ is not amenable, we will prove that $c|(A^{(B)}\times X)$ is essentially trivial.

We first claim that $c|(A_e\times X)$ is essentially trivial. 
Since $c$ is not amenable, $c$ is not elementary by Corollary \ref{amenelem}. 
Since the action $G\curvearrowright (X,\mu)$ is ergodic, Theorem \ref{CR} gives a partition into measurable $A^{(B)}$-invariant sets $X=X_1\sqcup X_2$ such that $c|(A_e\times X_1)$ is essentially trivial and $c|(A^{(B\setminus\{e\})}\times X_2)$ is essentially trivial. If $g\in G$, then $gA^{(B\setminus\{e\})}g^{-1}=A^{(B\setminus\{\pi(g)\})}$ and therefore Claim \eqref{equiv} implies that $c|(A^{B\setminus\{\pi(g)\})}\times gX_2)$ is essentially trivial. If $g\in G\setminus A^{(B)}$, then $A_e\subset A^{(B\setminus\{\pi(g)\})}$ and hence $c|(A_e\times gX_2)$ is essentially trivial. Since $X_2$ is $A^{(B)}$-invariant,  $Y=\cup_{g\in G\setminus A^{(B)}}gX_2$ is $G$-invariant and $c|(A_e\times Y)$ is essentially trivial. Since the action $G\curvearrowright (X,\mu)$ is ergodic,  $\mu(Y)\in\{0,1\}$. If $\mu(Y)=1$, then $c|(A_e\times X)$ is essentially trivial. If 
$\mu(Y)=0$, then $\mu(X_2)=0$ and hence $\mu(X_1)=1$ and $c|(A_e\times X)$ is again essentially trivial.

Now, if $b\in B$, then $gA_eg^{-1}=A_b$, for any $g\in\pi^{-1}(\{b\})$. This implies that $c|(A_b\times X)$ is essentially trivial, for every $b\in B$. Applying Lemma \ref{product} below gives that $c|(A^{(F)}\times X)$ is essentially trivial, for every finite set $F\subset B$. Writing $A^{(B)}=\cup_nA^{(F_n)}$, where $(F_n)\subset B$ is an increasing sequence of finite subsets with $\cup_nF_n=B$, and using Proposition \ref{ess_triv_amenable}, it follows that $c|(A^{(B)}\times X)$ is amenable.

By Lemma \ref{normalizer}, since  $A^{(B)}\lhd G$ and the action $G\curvearrowright (X,\mu)$ is ergodic, either $c|(A^{(B)}\times X)$ is essentially trivial, or $c|(A^{(B)}\times Z)$ is not essentially trivial, for any non-null measurable $A^{(B)}$-invariant set $Z\subset X$. In the latter case,
Corollary \ref{amenelem} implies that $c|(A^{(B)}\times X)$ is elementary.
 Proposition \ref{norm} further gives that $c|(G\times X)$ is elementary, which contradicts our assumption and finishes the proof.  
\end{proof}

\begin{lemma}\label{product}
Let $G_1,G_2<G$ and let $H$ be a countable group such that $G_2\subset N_G(G_1)$. Let $G\curvearrowright (X,\mu)$ be a pmp action and $c:G\times X\rightarrow H$ be a cocycle such that $c|(G_1\times X)$ and $c|(G_2\times X)$ are essentially trivial. Then $c|(G_1G_2\times X)$ is essentially trivial. 
\end{lemma}

\begin{proof}
It suffices to prove that given a non-null $G_1G_2$-invariant measurable subset $Y\subset X$,  there is a non-null $G_1G_2$-invariant measurable subset $Z\subset Y$ such that $c|(G_1G_2\times X)$ is essentially trivial.  Since $c|(G_i\times Y)$ is essentially trivial, we can find finite sets $F_i\subset H$ such that $$\text{$\mu(\{x\in Y\mid c(g,x)\in F_i\})>\frac{2\mu(Y)}{3}$, for all $g\in G_1$ and $i\in\{1,2\}$.}$$
Using this fact and the cocycle identity and denoting $F=F_1F_2$ it is immediate that 
$$\text{$\mu(\{x\in Y\mid c(g,x)\in F\})>\frac{\mu(Y)}{3}$, for all $g\in G_1G_2$.}$$
We now proceed as in the proof of \cite[Theorem 4.1]{BDV}. Consider the unitary representation $\pi:G\rightarrow\mathcal U(\text{L}^2(X\times H))$ given by $(\pi(g)^*\xi)(x,h)=\xi(gx,c(g,x)h)$, for all $g\in G$, $\xi\in\text{L}^2(X\times H)$, $x\in X$ and $h\in H$. Then for every $g\in G_1G_2$ we have that $$\text{$\langle \pi(g)^*(1_Y\otimes 1_F),1_Y\otimes 1_{\{e\}})=\mu(\{y\in Y\mid c(g,x)\in F\})>\frac{\mu(Y)}{3}$, for every $g\in G_1G_2$.} $$
From this it follows that  the element $\eta\in \text{L}^2(X\times H)$  of minimal $\|\cdot\|_2$ in the convex hull of $\{\pi(g)^*(1_Y\otimes 1_F)\mid g\in G_1G_2\}$ is non-zero and $\pi(G)$-invariant. Thus, we get a Borel map $\zeta:X\rightarrow\ell^2(H)$ such that $\zeta(gx)=c(g,x)\zeta(x)$, for all $g\in G_1G_2$ and a.e. $x\in X$.
 Let $Z=\{x\in Y\mid \zeta(x)\not=0\}$. For $x\in Z$, let $\varphi(x)=\{h\in H\mid \zeta(x)(h)=\|\zeta(x)\|_\infty\}$. Then $Z$ is non-null and $G_1G_2$-invariant. Moreover,  $\varphi(x)\subset H$ is finite and $\varphi(gx)=c(g,x)\varphi(x)$, for all $g\in G_1G_2$ and a.e. $x\in Z$. Thus, $c|(G_1G_2\times Z)$ is essentially trivial, as desired.
\end{proof}

\section{An extension of results of Monod-Shalom} 

The goal of this section is to extend certain results of Monod and Shalom from \cite{MS06} to infinite direct sums and to ICC groups in the class $\mathcal{C}$.
A countable group $\Gamma$ belongs to the class $\mathcal C$  \cite[Definition 2.15]{MS06} if it admits a mixing unitary representation $\pi$ such that $\text{H}^2_{\text{b}}(\Gamma,\pi)\not=0$.
The following ``cocycle reduction lemma'' allows us to extend the results in \cite{MS06} so that they apply not just to torsion-free groups in $\mathcal{C}$, but, more generally (since by \cite[Proposition 7.11]{MS06} any torsion-free group in $\mathcal C$ is ICC), to all  ICC groups in $\mathcal{C}$. 

\begin{lemma}\label{lem:cocyclered}
Let $G$ and $H$ be  countable groups with  normal subgroups $M\triangleleft G$ and $N\triangleleft H$. Let $G \curvearrowright (X,\mu)$ be a pmp action and let $H\curvearrowright Z$ be a smooth Borel action. Assume that:
\begin{enumerate}
\item we have an ME cocycle $c : G \times X\rightarrow H$, along with a Borel map $f:X\rightarrow Z$ with $f(g\cdot x)=c (g,x)\cdot f(x)$, for every $g\in M$ and almost every $x\in X$,
\item for all $z\in Z$, the projection of the stabilizer $H_z$ of $z\in Z$ to $H/N$ is finite,
\item $H/N$ has no nontrivial finite invariant random subgroups.
\end{enumerate}
Then there exists  a Borel map $\varphi : X\rightarrow H$ such that $\varphi (g\cdot x)c (g,x)\varphi (x)^{-1}\in N$, for every $g\in M$ and almost every $x\in X$.
\end{lemma}

This result holds more generally for locally compact second countable groups (with ``finite'' in items (2) and (3) replaced by ``compact''), but we state it in the case of countable groups, since this is the only case needed later on.
An {\it invariant random subgroup (IRS)} \cite{AGV} of a  countable group $H$ is a conjugation-invariant Borel probability measure $\mu$ on the  compact space $\text{Sub}_H$ of subgroups of $H$. 
An IRS  $\mu$  of $H$ is {\it finite} if $\mu$-a.e. $H_0\in\text{Sub}_H$ is finite, and {\it trivial} if it is supported on the trivial subgroup of $H$. 
We note that the assumption that $H/N$ has no nontrivial finite IRS is rather mild, and holds for instance if $H/N$ has a trivial amenable radical  \cite{BDL16} or is  ICC.

\begin{proof}
Let $\sigma _G$ and $\sigma _H$ be the counting measures on $G$ and $H$, respectively. Consider the action $G\times H \curvearrowright (X\times H, \mu\times \sigma _H)$ given by $g(x,h_0)=(g\cdot x,c(g,x)h_0)$ and $h(x,h_0)=(x,h_0h^{-1})$, for $g\in G, h\in H$. Since $c$ is an ME-cocycle,  there exist a pmp action $H\curvearrowright (Y,\nu )$, a cocycle $d :H\times Y\rightarrow G$, and a measure 
preserving isomorphism $\Phi : (Y\times  G, \nu\times \sigma_G)\rightarrow (X\times H, \mu\times \sigma _H)$ that is $(G\times H)$-equivariant, where $Y\times G$ is equipped with the $(G\times H)$-action given by $g(y,g_0)=(y,g_0g^{-1})$ and $h(y,g_0)=(h\cdot y,d(h,y)g_0)$, for $g\in G,h\in H$.

For $y\in Y$, let $p(y)\in X$ and $\rho(y)\in H$ be such that $\Phi (y,e)=(p(y),\rho(y) )$. Then for each $h\in H$ and $y\in Y$ we have
\begin{align*}
(p(h\cdot x),\rho({h\cdot y})h) &= h^{-1} \Phi (h\cdot y,e) \\&= \Phi (y, d(h,y)^{-1}) \\ &= d (h, y)\Phi (y, e) \\
&=  (d (h,y)\cdot p(y), c (d(h,y),p(y)) \rho({y})),
\end{align*}
so that 
$p(h\cdot y)=d (h,y)p(y)$ and $c (d(h,y),p(y)) = \rho({h\cdot y})h\rho(y)^{-1}$.

Let $\pi : (X,\mu) \rightarrow (W, \zeta)$ be the ergodic decomposition map for the action $M\curvearrowright (X,\mu)$. Since $M$ is normal in $G$, the action of $G$ on $(X,\mu)$ descends through $\pi$ to a pmp action of $G$ on $(W,\zeta)$. Applying Zimmer's cocycle reduction lemma \cite[Lemma 5.2.11]{Zi84} measurably across the $M$-ergodic components on $(X,\mu)$, we obtain measurable maps $\theta : W\rightarrow Z$ and $\varphi : X\rightarrow H$ such that $c ^{\varphi}(g,x)\coloneqq \varphi (g\cdot x)c(g,x)\varphi (x)^{-1} \in H_{\theta (\pi (y))}$ for all $g\in M$, $x\in X$. Let $h\mapsto \bar{h}$ denote the projection map $H\rightarrow H/N$ so that $\bar{H} = H/N$. For a map $\chi$ taking values in $H$, denote by  $\bar{\chi}$ its composition with the projection to $\bar{H}$, and for a subgroup $H_0<H$ denote $\bar{H}_0=\{\bar{h}\mid h\in H_0\}$. For $w\in W$, let $\bar{c}^{\varphi}_w : M\times \pi ^{-1}(w)\rightarrow \bar{H}_{\theta (w)}$ denote the restriction of $\bar{c}^{\varphi}$ to $M\times \pi ^{-1}(w)$. Since the groups $\bar{H}_{\theta (w)}$ are finite, by conjugating $c ^{\varphi}$ further if necessary we may assume without loss of generality that for almost every $w\in W$ the cocycle $\bar{c}^{\varphi}_w$
 is minimal, i.e., it takes values in a subgroup $K_w<\bar{H}_{\theta (w)}$ and is not cohomologous to a cocycle taking values in a proper subgroup of $K_w$.


Our next goal is to use the subgroups $(K_w)_{w\in W}$ to construct a finite  IRS of $\bar{H}$.
Let $\sigma _{\bar{H}}$ be the counting measure of $\bar{H}$. Consider the 
action $G\curvearrowright (X\times \bar{H},\mu\times \sigma _{\bar{H}} )$ given by $g(x,\bar{h})=(g\cdot x,\bar{c}^{\varphi}(g,x)\bar{h} )$. Then an ergodic decomposition for the restriction of this action to $M$ is obtained as follows: for almost every $w\in W$ and each $\bar{h}\in \bar{H}$, since the cocycle $\bar{c}^{\varphi}_w$ is minimal, the action of $M$ on $(\pi ^{-1}(w)\times K_w\bar{h} , \mu_w\times \sigma _{K_w\bar{h}})$ is ergodic, where $\mu_w$ is the measure on $\pi ^{-1}(w)$ coming from the disintegration of $\mu$ over $\pi$, and $\sigma _{K_w\bar{h}}$ is the right translation by $\bar{h}$ of the counting measure $\sigma_{K_w}$ of $K_w$. Let $\nu _w$ be the pushforward of $\sigma _{\bar{H}}$ to the right coset space $K_w\backslash \bar{H}$ under the map $\bar{h}\mapsto K_w\bar{h}^{-1}$. Then  $\sigma _{\bar{H}} = \int _{K_w\backslash \bar{H}} \sigma _{K_w\bar{h}}\; d\nu _w (K_w\bar{h})$. Hence,
\[
\mu\times \sigma _{\bar{H}} = \int _W\int _{K_w\backslash \bar{H}} (\mu_w \times \sigma _{K_w\bar{h}}) \; d\nu _w (K_w\bar{h})\, d\zeta(w)
\]
which exhibits an ergodic decomposition for the action of $M$ on $X\times\bar{H}$. 

Since $M$ is normal in $G$, the action of $G$ on $X\times\bar{H}$ permutes the $M$-ergodic components. Therefore, for every $g\in G$, for almost every $w\in W$ and $\bar{h}_0\in \bar{H}$ there is some $\bar{h}\in \bar{H}$ (depending on $g$, $w$, and $\bar{h}_0$) such that the measures $g_*(\mu_w \times \sigma _{K_w\bar{h}_0})$ and $\mu_{g\cdot w}\times \sigma _{K_{g\cdot w}\bar{h}}$ are equivalent. Thus, for $\mu_w$-a.e. $x\in \pi ^{-1}(w)$ the measures $\bar{c}^{\varphi}(g,x)_*\sigma _{K_w\bar{h}_0}$ and $\sigma _{K_{g\cdot w}\bar{h}}$ have the same support, i.e., $\bar{c}^{\varphi}(g,x)K_w\bar{h}_0 = K_{g\cdot w}\bar{h}$. For such $x$ we have $\bar{c}^{\varphi}(g,x)\in K_{g\cdot w}\bar{h}\bar{h}_0^{-1}$, so $\bar{c}^{\varphi}(g,x)K_w = K_{g\cdot w}\bar{h}\bar{h}_0^{-1} = K_{g\cdot w}\bar{c}^{\varphi}(g,x)$ and hence
\begin{equation}\label{K_yy}
\text{$K_{\pi (g\cdot x)} = \bar{c}^{\varphi}(g,x)K_{\pi (x)}\bar{c}^{\varphi}(g,x)^{-1}$, \;\;\; for a.e. $x\in X$.}
\end{equation}

Letting $C_x\coloneqq \varphi (x)^{-1}K_{\pi(x)}\varphi (x)$, we then have $C_{g\cdot x} = c(g,x)C_x c (g,x)^{-1}$ for every $g\in G$ and almost every $x\in X$. For $y\in Y$, define $C^y \coloneqq \rho(y) ^{-1}C_{p(y)} \rho(y)$. Since $p(h\cdot y)=d(h,y)p(y)$  and $c (d (h,y),p(y)) = \rho({h\cdot y})h\rho(y)^{-1}$, for every $h\in H$, we have
\begin{align*}
C^{h\cdot y} = \rho(h\cdot y)^{-1}C_{p(h\cdot y)}\rho(h\cdot y) &= \rho(h\cdot y)^{-1}C_{d(h,y)\cdot p(y)}\rho(h\cdot y)  \\ &= \rho(h\cdot y) ^{-1}c(d(h,y),p(y))C_{p(y)}c(d(h,y),p(y))^{-1}\rho(h\cdot y) \\
&=  h \rho(y) ^{-1}C_{p(y)}\rho(y) h^{-1} \\
&= hC^y h^{-1}.
\end{align*}

Therefore, the image of $\nu$ under $y\mapsto C^y$ is a finite IRS of $\bar{H}$. Hence $C^y$ is trivial almost surely since by assumption $\bar{H}$ has no nontrivial finite IRS. This implies that $K_{\pi(p(y))}$ is trivial for a.e. $y\in Y$.
Since $G\cdot\Phi(Y\times\{e\})=\Phi(G\cdot (Y\times\{e\})=\Phi(Y\times G)=X\times H$ and $G\cdot \Phi(Y\times \{e\})\subset G\cdot (p(Y)\times H)\subset G\cdot p(Y)\times H$, we get that $G\cdot p(Y)\subset X$ is $\mu$-conull.
By combining the last two facts with \eqref{K_yy}, we further deduce that
 $K_{\pi(x)}$ is trivial for a.e. $x\in X$. 
 
 Hence $K_w$ is trivial for a.e. $w\in W$. Thus, $\varphi (g\cdot x) c(g,x)\varphi (x)^{-1} \in N$, for every $g\in M$ and almost every $x\in X$.
\end{proof}

In what follows, given a direct sum $H = \bigoplus _{j\in J}H _j$ and $j_0\in J$, we write $\widehat{H}_{j_0}$ for $\bigoplus _{j\neq j_0}H _j$. In the case where $J$ is finite and $G/M$ is assumed torsion-free, the following is shown in \cite[Proposition 5.1]{MS06}. We extend this result to the case where $J$ is countable 
 (i.e., either finite or countably infinite) and $G/M$ is only assumed ICC.

\begin{proposition}\label{prop:classCquotient}
Let $(\Omega, m)$ be an ME-coupling of countable groups $G$ and $H = \bigoplus _{j\in J} H_j$. Suppose that $M$ is a normal subgroup of $G$ such that $G/M$ is ICC and belongs to the class $\mathcal{C}$. Then there exist $j_0 \in J$ and a $G$-fundamental domain $Y\subset \Omega$ such that $\widehat{H} _{j_0} Y\subset MY$.
\end{proposition}

\begin{proof}
Let $p:G\rightarrow G/M$ be the quotient homomorphism and let $\pi: G/M\rightarrow\mathcal U(\mathcal H)$ be a mixing unitary representation with $\text{H}^2_{\text{b}}(G/M , \pi )\neq 0$. Fix measurable fundamental domains $Y_0, X_0\subset \Omega$ for $G, H$ respectively. 
After possibly replacing $Y_0$ by $g Y_0$, for some $g\in G$, we may assume that $Y_0\cap X_0$ is non-negligible.
Let  $\mu=m(X_0)^{-1}m_{|X_0}$ and let $G\curvearrowright (X_0,\mu)$  be the associated pmp action of $G$ given by $\{g \cdot x\} = G x\cap X_0$, and likewise denote the pmp action $H\curvearrowright (Y_0,\nu)$ by $h \cdot y$. Let $c:H\times Y_0\rightarrow G$ be the associated ME-cocycle given by $c(h,y)hy =h \cdot y$. 


Assume first that $J$ is finite. If $\widetilde\pi=\pi\circ p:\Gamma\rightarrow\mathcal U(\mathcal H)$, then $\text{H}^2_{\text{b}}(\Gamma ,\widetilde\pi)\neq 0$ by \cite[Corollary 3.6]{MS06}. Arguing as in the first part of the proof of \cite[Proposition 5.1]{MS06}, we  find some $j_0 \in J$ and a nonzero measurable function $f:Y_0 \rightarrow \mathcal{H}$ such that $f(h \cdot y) = \widetilde\pi (c(h,y))f(y)$ for all $h\in\widehat{H}_{j_0}$ and almost every $y\in Y_0$. Since $\pi$ is mixing, $\widetilde\pi$ is smooth. Since $G /M$ is ICC,  we can apply Lemma \ref{lem:cocyclered} to obtain a Borel map $\varphi :Y_0\rightarrow G$ with $\varphi (h\cdot y)c(h,y)\varphi (y)^{-1}\in M$ for all $h \in \widehat{H}_{j_0}$ and almost every $y\in Y_0$. Then $Y \coloneqq \{ \varphi (y)y \mid y\in Y_0 \}$ satisfies the conclusion. 
Indeed, for every $\in\widehat{H}_{j_0}$ and almost every $y\in Y_0$,  we have
$$h(\varphi(y)y)=\varphi(y)hy=(\varphi(h\cdot y)c(h,y)\varphi(y)^{-1})^{-1}\varphi(h\cdot y)(h\cdot y)\in MY.$$
Moreover,  for further reference, we note that the cocycle $(p\circ c)_{|\widehat{H}_{j_0}\times Y_0}$ is cohomologous to the trivial homomorphism $\widehat{H}_{j_0}\rightarrow G/M$, hence amenable.

Assume now that $J$ is infinite. 
For a subset $I$ of $J$, we denote by $H_{(I)}$ the subgroup of $H$ comprising elements whose projection to each coordinate outside of $I$ is the identity.
We claim that there is a finite subset $F$ of $J$ such that the cocycle $(p\circ c)_{|H_{(F)}\times Y_0}$  is not amenable.
Otherwise,  Proposition \ref{ess_triv_amenable}(2) implies that the cocycle $p\circ c:H\times Y_0\rightarrow G/M$ is amenable.
By Proposition \ref{ess_triv_amenable}(6), this gives that $G/M$ is amenable, which is a contradiction.


If a finite subset $F\subset J$  is such that the cocycle $(p\circ c)_{|H_{(F)}\times Y_0}$  is not amenable, 
the conclusion follows by applying the finite case to the finite direct sum decomposition $\Lambda = \Lambda _{(J\setminus F)}\oplus \bigoplus _{j\in F}\Lambda _j$.
\end{proof}

\begin{proposition}\label{prop:tsurj}
Let $(\Omega ,m)$ be an ME-coupling of countable groups $G = \bigoplus _{i\in I}G _i$ and $H = \bigoplus _{j\in J} H _j$ with $I$ and $J$ countable. Assume that each $G _i$ is an ICC group in the class $\mathcal{C}$. Then there exists a map $t:I\rightarrow J$ such that for each finite subset $F$ of $I$ there exists a fundamental domain $Y_F\subset \Omega$ for the action of $G$ such that $\widehat{H}_{t(i)}Y_F\subset \widehat{G}_i Y_F$ for all $i\in F$. Such a map $t$ is surjective provided that either (1) $I$ is finite and each $H _j$ is infinite, or (2) $I$ is infinite and each $H _j$ is nonamenable.

Moreover, if each $H_j$ is an ICC group in the class $\mathcal{C}$, then there exists a unique such map $t$, which is bijective, and for each finite subset $Q$ of $J$ there exists a fundamental domain $X_Q\subset \Omega$ for the action of $H$ such that $\widehat{G}_{t^{-1}(j)}X_Q\subset \widehat{H}_j X_Q$ for every $j\in Q$.
\end{proposition}

\begin{proof}
The existence of a not necessarily surjective map $t:I\rightarrow J$ with the stated property follows from Proposition \ref{prop:classCquotient} and the inductive step in the proof of \cite[Proposition 5.1]{MS06}. The fact that $t$ is surjective under assumption (1) follows exactly as in the proof of surjectivity in \cite[Proposition 5.1]{MS06} (which requires the assumption that each $H _j$ is infinite). 

Suppose that (2) holds, and assume the notation from the proof of Proposition \ref{prop:classCquotient}. Since $H_j$ is nonamenable and $c$ is an ME-cocycle, the cocycle $c{|(H_j\times Y_0)}$ is not amenable by Proposition \ref{ess_triv_amenable}(4). Proposition \ref{ess_triv_amenable}(4) further provides a finite set $F\subset I$ such that the cocycle $q\circ c:H_j\times Y_0\rightarrow G_{(F)}$ is not amenable, where $q:G\rightarrow G_{(F)}$ is the quotient homomorphism.

We claim that $j\in t(F)$. Otherwise, we have
\[
\Lambda _{j}Y_F\subset \bigcap _{i\in F}\widehat{H}_{t(i)}Y_F \subset \bigcap _{i\in F}\widehat{G}_i Y_F = G _{(I\setminus F)}Y_F,
\]
so there exists a measurable map $\varphi : Y_0 \rightarrow G$ such that $\varphi (h \cdot y)c(h,y)\varphi (y)^{-1} \in G _{(I\setminus F)}$ for all $h \in H _{j}$ and a.e. $y\in Y_0$.
Hence $q(\varphi (h \cdot y))q(c (h,y))q(\varphi (y)^{-1})=0$, for all $h\in H_j$ and a.e. $y\in Y_0$. This shows that the cocycle $q\circ\beta:H_j\times Y_0\rightarrow G_{(F)}$ is essentially trivial, which is a contradiction.
The map $t$ is therefore surjective.

The proof of the last statement of the proposition is similar to the second paragraph of the proof of \cite[Theorem 1.16]{MS06}, and does not depend on the surjectivity that we already proved above.  Assume that each $H _j$ is an ICC group in the class $\mathcal{C}$. Then we can find a  map $s:J\rightarrow I$ such that for each finite $Q\subset J$ there is a fundamental domain $X_Q\subset \Omega$ for the action of $H$ with $\widehat{G}_{s(j)}X_Q\subset \widehat{H}_j X_Q$ for all $j\in Q$. 

We claim that $t(s(j))=j$, for every $j\in J$.
Given $j\in J$, we find  fundamental domains $X,Y\subset\Omega$ for the actions of $H,G$, respectively, with
$\widehat{G} _{s(j)}X \subset \widehat{H}_jX$ and 
$\widehat{H}_{t(s(j))}Y\subset \widehat{G}_{s(j)}Y$. 
By translating $Y$ by an element of $G$ if necessary, we may assume without loss of generality that the intersection $A=X\cap Y$ is non-negligible. To show that $t(s(j))=j$, it is enough to show that each infinite subgroup $K<\widehat{H}_{t(s(j))}$ has nontrivial intersection with $\widehat{H}_j$. Indeed, this implies that $H _j$ is not contained in $\widehat{H}_{t(s(j))}$ and hence $t(s(j))=j$. 

Given an infinite subgroup $K$ of $\widehat{H}_{t(s(j))}$, by Poincar\'{e} recurrence there is some nontrivial $h\in K$ such that $h\cdot A\cap A$ has positive measure, and hence there is some $g \in G$ such that $hA\cap g A$ has positive measure. Then
\[
hA\cap g A \subset \widehat{H}_{t(s(j))} Y\cap gY \subset \widehat{G}_{s(j)}Y\cap g Y,
\]
and since $Y$ is a fundamental domain for the action of $G$ it must be that $g \in \widehat{G}_{s(j)}$. Therefore,
\[
hA\cap g A \subset hX \cap \widehat{G}_{s(j)} X\subset hX \cap \widehat{H}_j X.
\]
As $X$ is a fundamental domain for the action of $H$ it must be that $h\in \widehat{H}_j$, so $K\cap \widehat{H}_j$ is nontrivial, as claimed. 

We have shown that $t(s(j))=j$ for all $j\in J$. Switching the roles of $G$ and $H$ we  obtain $s(t(i))=i$ for all $i\in I$, so that $t$ is invertible with inverse $s$. This finishes the proof.
\end{proof}

\begin{theorem}\label{thm:MSinfiniteICCsum}
Suppose $G _i$, $i\in I$, and $H _j$, $j\in J$, are ICC groups in the class $\mathcal{C}$, where $I,J$ are countable. If $G \coloneqq \bigoplus _{i\in I}G _i$ and $H \coloneqq \bigoplus _{j\in J} H _j$ are measure equivalent then $|I|=|J|$ and there exists a bijection $t:I\rightarrow J$ such that $G _i$ is measure equivalent to $H_{t(i)}$ for all $i\in I$.
\end{theorem}

\begin{proof}
Let $(\Omega ,m)$ be an ME-coupling of $G$ and $H$. Let $t:I\rightarrow J$ be as in Proposition \ref{prop:tsurj}. By that Proposition, $t$ is bijective (hence $|I|=|J|$), and for each $i\in I$ we can find fundamental domains $Y_i,X_i\subset \Omega$ for the action of $G ,H$, respectively, such that $\widehat{H} _{t(i)}Y_i\subset \widehat{G}_iY_i$ and $\widehat{G}_iX_i\subset \widehat{H}_{t(i)}X_i$. Then $G _i \cong G /\widehat{G}_i$ and $H _{t(i)}\cong H /\widehat{H}_{t(i)}$ are measure equivalent by Proposition \ref{prop:coupling_subgroup_quotient}(2) below.
\end{proof}

\begin{proposition}\label{prop:coupling_subgroup_quotient}
Let $G$ and $H$ be countable groups, let $M\triangleleft G$ and $N\triangleleft H$ be normal subgroups, and let $\bar{G}=G/M$ and $\bar{H}= H/N$. Let $(\Omega ,m)$ be an ergodic ME-coupling of $G$ and $H$, and suppose that there exist finite measure fundamental domains $Y, X\subset \Omega$ for the actions of $G$ and $H$, respectively, which satisfy $NY\subset MY$ and $MX\subset NX$. Then
\begin{enumerate}
\item \cite{Ki14a} The groups $M$ and $N$ are measure equivalent. Moreover, after replacing $Y$ by its translate by some element of $G$, the action of $M\times N$ on $(MY\cap NX ,m_{|MY\cap NX} )$ is an ME-coupling of $M$ and $N$.
\item \cite{MS06} The groups $\bar{G}$ and $\bar{H}$ are measure equivalent. Moreover, if  $\mathcal{A}$ is the $\sigma$-algebra of all $(M\times N)$-invariant measurable subsets of $\Omega$, then up to rescaling by a constant there is a unique $(\bar{G}\times\bar{H})$-invariant $\sigma$-finite measure $\nu$ on $\mathcal{A}$ that is equivalent to $m_{|\mathcal{A}}$, and the action of $\bar{G}\times\bar{H}$ on $(\mathcal{A},\nu )$ is an ergodic ME coupling of  $\bar{G}$ and $\bar{H}$.
\end{enumerate}
\end{proposition}

\begin{proof}
(1) After replacing $Y$ by $gY$, for some $g\in G$, we may assume that $m(Y\cap X) >0$ and hence that $m(MY\cap NX)>0$. Let $\Omega _0 = MY\cap NX$ and let $m _0 = m _{|\Omega _0}$. Since $NY\subset MY$ and $MX\subset NX$, both $MY$ and $NX$ are $(M\times N)$-invariant, and hence so is $\Omega _0$. The set $X\cap MY$ has finite measure since $X$ does, and is a fundamental domain for the action of $N$ on $\Omega _0$. Indeed, $h (X\cap MY)$, $h \in N$, are pairwise disjoint since $X$ is a fundamental domain for the action of $N$ on $\Omega$, and cover $\Omega _0$ since $N(X\cap MY)=\bigcup _{h \in N}h X\cap h MY = \bigcup _{h\in N}h X\cap MY = \Omega _0$ by the $N$-invariance of $MY$. Similarly, $NX\cap Y$ is a finite measure fundamental domain for the action of $M$ on $\Omega _0$, so $(\Omega _0,m _0 )$ is an ME-coupling of $M$ and $N$.

(2) The ergodic $(G \times H)$-action on $(\Omega ,m)$ descends to an ergodic $(\bar{G}\times\bar{H})$-action on $(\mathcal{A},m_{|\mathcal{A}})$. The restricted measure $m_{|\mathcal{A}}$ is not $\sigma$-finite if $N$ (equivalently, $M$) is infinite, but by selecting a set $T_N\subset H$ of representatives for the right cosets of $N$ in $H$ we obtain a $(\bar{G}\times\bar{H})$-invariant $\sigma$-finite measure $\nu$ on $\mathcal{A}$ that is equivalent to $m_{|\mathcal{A}}$  given by
$\nu (Z)= \mu _{\Omega}(Z\cap T_NX)$,
for $Z\in \mathcal{A}$. 
Since $\nu$ is equivalent to $m_{|\mathcal{A}}$, the $(\bar{G}\times\bar{H})$-action on $(\mathcal{A},\nu )$ is ergodic, and hence, up to rescaling by a constant, $\nu$ is the unique $(\bar{G}\times\bar{H})$-invariant $\sigma$-finite measure that is equivalent to $m_{|\mathcal{A}}$. 
Since $MX\subset NX$, we get that $NX$ belongs to $\mathcal{A}$. Since $\nu (NX)=m(X)<\infty$, we also get that  $NX$ is a finite measure fundamental domain for the action of $\bar{H}$ on $(\mathcal{A}, \nu )$. 

On the other hand, selecting a set $T_M\subset G$ of representatives for the right cosets of $M$ in $G$ we get a $(\bar{G}\times\bar{H})$-invariant $\sigma$-finite measure $\nu '$ on $\mathcal{A}$ that is equivalent to $m_{|\mathcal{A}}$, defined by
$\nu ' (Z)= \mu _{\Omega}(Z\cap T_MY)$, for $Z\in\mathcal A$. 
As above we see that $MY$ is a $\nu '$-finite measure fundamental domain for the action of $\bar{G}$ on $(\mathcal{A}, \nu ' )$. Since $\nu '$ and $\nu$ are equivalent measures, $\nu '$ is a scalar multiple of $\nu$, so $\nu (MY)<\infty$. Thus, $(\mathcal{A},\nu )$ is an ME-coupling of $\bar{\Gamma}$ and $\bar{\Lambda}$.
\end{proof}

We end this section with the following consequence of Lemma \ref{lem:cocyclered}.

\begin{lemma}\label{lem:ME_ICC_quotient}
Let $G, H$ be countable groups with normal subgroups $M\triangleleft G, N\triangleleft H$. Let $(\Omega,m)$ be an ME-coupling of $G$ and $H$,  $Y_0\subset \Omega$ a fundamental domain for the action of $G$, and $c:H\times Y_0\rightarrow G$  the associated ME-cocycle. 
Let  $\bar{G}=G/M$ and denote by $\bar{c} :H\times Y_0\rightarrow \bar{G}$ the composition of $c$ with the projection map $g\mapsto\bar{g}$ from $G$ to $\bar{G}$. 
Assume that the cocycle $\bar{c}|(N\times Y_0)$ is essentially trivial and $\bar{G}$ is ICC.

Then there is a fundamental domain $Y\subset \Omega$ for the action of $G$ such that $NY\subset MY$. 
\end{lemma}

\begin{proof}
Consider the action of $G$ on  $\mathcal{F}(\bar{G})$ given by $g\cdot F=\bar{g}F$. Then each stabilizer of this action of $G$ projects to a finite subgroup of $\bar{G}$. Since $\bar{c}|(N\times Y_0)$ is essentially trivial, there exists a measurable map $f:Y_0\rightarrow \mathcal{F}(\bar{G})$ satisfying $f(h\cdot y) = c (h,y)\cdot f(y)$, for every $h\in N$ and almost every $y\in Y_0$. 

Since $\bar{G}$ is ICC, the hypothesis of Lemma \ref{lem:cocyclered} is satisfied, so there exists a Borel map $\varphi :Y_0\rightarrow G$ such that $\varphi (h\cdot y)c(h,y)\varphi (y)^{-1} \in M$, for every $h\in N$ and almost every $y\in Y_0$. Define $Y= \{ \varphi (y)y | y\in Y_0 \}$.
Then $Y$ is a fundamental domain for the action of $G$ on $\Omega$, and for  every $h\in H$ and almost every $y\in Y_0$ we have
\[
h\varphi (y)y = \varphi (y)hy=\varphi (y)c(h, y)^{-1}(h\cdot y) = (\varphi (h\cdot y)c(h,y)\varphi (y)^{-1})^{-1}\varphi (h\cdot y)(h\cdot y).
\]
Therefore, $NY\subset MY$, as claimed.
\end{proof}

\section{Pairwise non-ME groups in the class $\mathcal C$}

In this section, we establish another ingredient needed to prove Theorem \ref{A} by constructing an infinite family of pairwise non-ME groups in class $\mathcal C$ which can  be generated by a fixed number of generators. 

\begin{proposition}\label{prop:direct_sum_family}
There exists a continuum sized family $\{A_c\}_{c\in I}$  of finitely generated ICC groups in the class $\mathcal{C}$ which are pairwise not measure equivalent. 
\end{proposition}

\begin{proof}
In the proof of \cite[Proposition VI.16]{Ga00} Gaboriau constructs, for each irrational number $c\in (1,1+\tfrac{1}{8})$, a finitely generated group $\Gamma _c$ of fixed price $c$ which is an amalgam of the form  $\Gamma _c=J_c\ast _{K_c} H_c$, where the group $J_c$ is strongly treeable of fixed price $c$, both $H_c$ and $K_c$ are infinite amenable groups, and $K_c$ has infinite index in both $J_c$ and $H_c$. 
Thus, $\beta_n^{(2)}(H_c)=\beta_n^{(2)}(K_c)=0$, for every $n\geq 1$, $\beta_1^{(2)}(J_c)=c-1$ and $\beta_n^{(2)}(J_c)=0$, for every $n\geq 2$ (see \cite[Corollaire 3.23 and Proposition 6.10]{Ga02}).
From this and \cite{CG} it follows that $\beta ^{(2)}_1 (\Gamma _c ) = c-1$ and $\beta ^{(2)}_n (\Gamma _c)=\beta ^{(2)}_n(J _c) + \beta ^{(2)}_n(H_c) = 0$, for every $n\geq 2$. 

Let $\mathbb F_2$ denote the free group on two generators and let $A _c = \Gamma _c \ast (\mathbb F_2\times \mathbb F_2 )$. Then $\beta ^{(2)}_1(A _c)=c$, $\beta ^{(2)}_2(A _c)=1$, and $\beta ^{(2)}_n(A_c) = 0$ for $n\geq 3$. It follows from \cite{Ga02} that the groups $A _c$ are pairwise not measure equivalent. 
The groups $A _c$, each being a free product of infinite groups, are all ICC and belong to the class $\mathcal{C}$ (by \cite[Theorem 1.3]{MS06}). 
\end{proof}

Since the family $\{A_c\}_{c\in I}$ is uncountable, we can find  $n\geq 2$ and an uncountable subfamily $\{A_c\}_{c\in J}$ such that $A_c$ can be generated by $n$ elements, for every $c\in J$. This will be enough to prove the main assertion of Theorem \ref{A}. However, since the groups $A_c$ contain torsion,  in order to derive the moreover part of Theorem \ref{A} we will need the following:

\begin{proposition}\label{torsion-free-family}
There exists a countably infinite family $\{B_d\}_{d\geq 1}$  of torsion-free, $4$-generated ICC groups in the class $\mathcal{C}$ which are pairwise not measure equivalent. 
\end{proposition}

In preparation for the proof of this result, recall from \cite{Ga02} that the {\it approximate ergodic dimension} $\text{a-erg-dim}(\Gamma )$ of a countable group $\Gamma$ is the least $d\in \N \cup \{ \aleph _0 \}$ such that there exists a free pmp action $\Gamma \curvearrowright (X,\mu )$ whose orbit equivalence relation $\mathcal{R}$ may be expressed as an increasing union $\mathcal{R}=\bigcup _{i\in \N}\mathcal{R}_i$ of Borel subequivalence relations such that each $\mathcal{R}_i$ admits a discrete Borel simplicial action on a Borel bundle of contractible simplicial complexes of dimension $d$. 
The approximate ergodic dimension is a measure equivalence invariant.

\begin{proof}
Let $d\geq 1$ and consider the action of $\Z$ on the direct sum $\bigoplus _{i=1}^d \mathbb F_2$ of $d$ copies of $\mathbb F_2$  via a cyclic permutation of the direct factors. Let $C _d\coloneqq (\bigoplus_{i=1}^d \mathbb F_2)\rtimes \Z$ be the corresponding semi-direct product, which is a torsion-free, 3-generated group with approximate ergodic dimension $d$ by \cite{Ga02}.
Indeed, the inequality $\text{a-erg-dim}(C _d)\leq d$ is straightforward. We also have $d\leq \text{a-erg-dim}(\bigoplus _{i=1}^d \mathbb F_2)\leq \text{a-erg-dim}(C _d)$ where the first inequality follows from \cite[Corollaire 5.13]{Ga02} since $\beta ^{(2)}_d(\bigoplus _{i=1}^d \mathbb F_2)=1$, and the second inequality follows from $\bigoplus _{i=1}^d \mathbb F_2$ being a subgroup of $C_d$.
Thus, the groups $C _d$, $d\geq 1$,  are pairwise not measure equivalent.

Finally, let $B_d=C_d*\mathbb Z$, for every $d\geq 1$. Then $B_d\in \mathcal C$ by \cite[Theorem 1.3]{MS06} and $B_d$ is torsion-free and ICC for every $d\geq 1$. Moreover, since  $\beta_1^{(2)}(C_d)=0$ for every $d\geq 1$ by \cite[Th\'{e}or\`{e}me 6.8]{Ga02}, applying \cite[Theorem 1.5]{AG} implies that the groups $\{B_d\}_{d\geq 1}$ are pairwise not measure equivalent.
\end{proof}

\section{Proof of Theorem \ref{A}}

We are now ready to prove Theorem \ref{A}. We first establish the following:

\begin{theorem}\label{descend}
Let $A$ and $C$ be any countable groups,  and let $B$ and $D$ be ICC nonamenable subgroups of hyperbolic groups.
Assume that $G\in\mathcal W\mathcal R(A,B)$ and $H\in \mathcal{W}\mathcal{R}(C,D)$ are measure equivalent groups. Then the groups $A^{(B)}$ and $C^{(D)}$ are measure equivalent.
\end{theorem}

\begin{proof}
Let $(\Omega ,m)$ be an ergodic ME-coupling of $G$ and $H$, $Y_0\subset \Omega$ a fundamental domain for the action of $G$ and let $H\curvearrowright (Y_0,\nu)$ be the associated ergodic pmp action, where $\nu=m(Y_0)^{-1}m_{|Y_0}$.
Let $c:H\times Y_0\rightarrow G$ be the associated ME-cocycle and $\bar{c}:H\times Y_0\rightarrow G/A^{(B)}$ be the composition of $c$ with the projection from $G$ to $G/A^{(B)}\cong B$. 

Since $B$ is a subgroup of a hyperbolic group, Corollary \ref{CR2} implies that either the cocycle $\bar{c}$ is amenable or  the cocycle $\bar{c}|(C^{(D)}\times Y_0)$ is essentially trivial. Since $c$ is an ME-cocycle, by Proposition \ref{ess_triv_amenable}(6), $c$ being amenable would imply that $B$ is amenable, which is a contradiction. Thus, we must have that the cocycle $\bar{c}|(C^{(D)}\times Y_0)$ is essentially trivial.
Hence, by using Lemma \ref{lem:ME_ICC_quotient} and that $B$ is ICC, we find a fundamental domain $Y\subset \Omega$ for the action of $G$ such that $C^{(D)}Y\subset A^{(B)}Y$.
Likewise, there is a fundamental domain $X\subset \Omega$ for the action of $H$ with $A^{(B)}X\subset C^{(D)}X$. Therefore, $A^{(B)}$ and $C^{(D)}$ are measure equivalent by Proposition \ref{prop:coupling_subgroup_quotient}.
\end{proof}

\begin{proof}[\bf Proof of Theorem \ref{A}] Proposition \ref{prop:direct_sum_family} implies that there exist $n\geq 2$ and an infinite family $\{A_k\}_{k=1}^\infty$ of $n$-generated ICC groups in class $\mathcal C$ which are pairwise not measure equivalent.  

For $x\in\{0,1\}^\mathbb N$, define $I_x=\{i\in\mathbb N\mid x(i)=1\}$ and $D_x=\bigoplus_{i\in I_x}A_i$. Since $A_i$ is $n$-generated, for every $i\geq 1$, Proposition \ref{WR} provides an ICC nonamenable normal subgroup $B$ of a torsion-free hyperbolic group such that for every $x\in\{0,1\}^{\mathbb N}$ we can find a property (T) group $G_x\in\mathcal W\mathcal R(D_x,B)$.


We claim that  the groups $\{G_x\}_{x\in\{0,1\}^{\mathbb N}}$ are pairwise non-measure equivalent. Assume that $G_x$ and $G_y$ are measure equivalent, for some $x,y\in\{0,1\}^{\mathbb N}$. By our choice of $B$, Theorem \ref{descend} implies that $D_x^{(B)}\cong\oplus_{i\in I_x}A_i^{(\mathbb N)}$ and $D_y^{(B)}\cong \oplus_{i\in I_y}A_i^{\mathbb N}$ are measure equivalent.
Since $A_k$ is ICC and in the class $\mathcal C$, for every $i\geq 1$, Theorem \ref{thm:MSinfiniteICCsum} 
implies that the sets of groups $\{A_i\}_{i\in I_x}$ and $\{A_i\}_{i\in I_y}$ coincide up to measure equivalence.
Since the groups $\{A_i\}_{i\geq 1}$ are pairwise not measure equivalent, we get $I_x=I_y$ hence $x=y$, which proves the claim and the main assertion of Theorem \ref{A}.

Towards proving the moreover assertion, let $x\in\{0,1\}^{\mathbb N}$. The K\"{u}nneth formula implies that $\beta_p^{(2)}(D_x^{(B)})=0$, for every $p\geq 1$. Using this, the fact that we have a short exact sequence $\{e\}\longrightarrow D_x^{(B)} \longrightarrow G_x \longrightarrow B\longrightarrow \{e\}$ and that $B$ has an element of infinite order, it follows from \cite[Theorem 7.2(6)]{Luck} that $\beta_p^{(2)}(G_x)=0$, for every $p\geq 1$.

Finally, we note that by using Proposition \ref{torsion-free-family} instead of Proposition \ref{prop:direct_sum_family}, we may assume that the groups $\{A_i\}_{i\geq 1}$ are torsion-free. Thus, if $x\in\{0,1\}^{\mathbb N}$, then $D_x$ is torsion-free. Further, since $G_x\in\mathcal W\mathcal R(D_x,B)$ and $B$ is torsion-free, we deduce that $G_x$ is torsion-free.
\end{proof}

\section{Measure equivalence and isomorphism on the space of marked finitely generated groups}\label{sec:Borel} 

\subsection{Borel reducibility}

Let $\mathcal{R}$ and $\mathcal{S}$ be equivalence relations on standard Borel spaces $X$ and $Y$, respectively. 
The equivalence relation $\mathcal{R}$ is said to be \emph{Borel reducible} to $\mathcal{S}$, denoted $\mathcal{R}\leq_B\mathcal{S}$, if there exists a Borel map $f:X\to Y$ such that $(x_0,x_1)\in \mathcal{R}$ if and only if $(f(x_0),f(x_1))\in \mathcal{S}$; such a map $f$ is called a \emph{Borel reduction} from $\mathcal{R}$ to $\mathcal{S}$.
The relation $\leq_B$ defines a preorder on such equivalence relations.
An equivalence relation $\mathcal{R}$ on a standard Borel space is called \emph{smooth} if it is Borel reducible to the equality relation on $\mathbb{R}$.

An equivalence relation $\mathcal{R}$ on a standard Borel space $X$ is called Borel if it is Borel as a subset of $X\times X$ (equipped with the product $\sigma$-algebra). 
$\mathcal{R}$ is called \emph{countable} if every $\mathcal{R}$-class is countable. 
A countable Borel equivalence relation $\mathcal{S}$ is said to be \emph{(countable) universal} if $\mathcal{R}\leq_B\mathcal{S}$ for every countable Borel equivalence relation $\mathcal{R}$. 
By \cite[Proposition 1.8]{DJK} the orbit equivalence relation $\mathcal{R}(\mathbb{F}_2\curvearrowright \{ 0,1\}^{\mathbb{F}_2})$ associated with the $2$-shift of the rank two free group is countable universal.

If $\mathcal{S}$ is a universal countable Borel equivalence relation then $\mathcal{S}$ is non-smooth, as is any equivalence relation $\mathcal{R}$ with $\mathcal{S}\leq_B\mathcal{R}$.
Informally speaking, such an $\mathcal{S}$ is in fact very far from being smooth; see \cite{KecCBER} for a comprehensive survey covering the $\leq_B$-preorder for countable Borel equivalence relations.

In this section we are interested in the Borel reducibility complexity of measure equivalence and isomorphism on the space of marked finitely generated groups.
While all equivalence classes of the isomorphism relation are countable, the same is not true for measure equivalence: for example, the set of all (marked, finitely generated) infinite amenable groups forms a single measure equivalence class of cardinality continuum.
Regardless, our goal is to prove that every countable Borel equivalence relation is Borel reducible to both of these relations on a certain subspace of groups.

\subsection{The space of marked finitely generated groups} For a group $G$, let $\mathrm{NSub}(G)$ be the space of all normal subgroups of $G$, equipped with the topology of pointwise convergence of indicator functions, which is compact metrizable.
We write $\mathrm{Quot}(G)$ for the space of all quotient groups of $G$, equipped with the topology making the group quotient map $\mathrm{NSub}(G)\to\mathrm{Quot}(G)$, $K\mapsto G/K$, a homeomorphism.

Let $\mathbb{F}_n$ be a free group, freely generated by $a_0,\dots ,a_{n-1}$, and let $\varphi_n :\mathbb{F}_{n+1}\to\mathbb{F}_{n}$ be the homomorphism that is the identity map on $a_0,\dots , a_{n-1}$ and sends $a_n$ to the identity element.  
This induces a topological embedding $\mathrm{Quot}(\mathbb{F}_{n})\to\mathrm{Quot}(\mathbb{F}_{n+1})$, $\mathbb{F}_{n}/K\mapsto \mathbb{F}_{n+1}/\varphi_n^{-1}(K)$, whose image is clopen.
The \emph{space of marked finitely generated groups}, denoted $\mathcal{G}_{\mathrm{fg}}$, is the associated direct limit of the spaces $\mathrm{Quot}(\mathbb{F}_n)$, $n\in\N$.

If $G$ is a finitely generated group then, by choosing a surjective homomorphism $\varphi :\mathbb{F}_{n}\to G$, we obtain a topological embedding $\mathrm{Quot}(G)\hookrightarrow\mathrm{Quot}(\mathbb{F}_{n})\subseteq\mathcal{G}_{\mathrm{fg}}$, $G/K\mapsto \mathbb{F}_{n}/\varphi^{-1}(K)$, that sends each group in $\mathrm{Quot}(G)$ to an isomorphic copy in $\mathrm{Quot}(\mathbb{F}_n)$.

\subsection{A Borel version of Theorem \ref{A}}

We will use the following corollary to the proof of Theorem \ref{A}.

\begin{corollary}\label{Borelfam}
There exists a property (T) group $V_0$ and a continuum sized Borel family $(G_x)_{x\in \{ 0, 1\}^{\N}}$ of quotients of $V_0$ whose members are pairwise not measure equivalent.
Each group $G_x$ is an ICC, torsion-free group with property (T).

Moreover, the Borel family $(G_x\ast \mathbb{Z} )_{x\in \{ 0,1\}^{\N}}$ consists of ICC, torsion-free groups in the class $\mathcal{C}$, which are pairwise not measure equivalent.
\end{corollary}

\begin{proof}
The proof of Theorem \ref{A} already produces the desired family, but in order to see  that the construction is Borel we need to be more explicit about how this family is obtained.

We fix, by Proposition \ref{WR}, a property (T) group $V_0$ with $V_0\in\mathcal{WR}(F^{(\N )},D_0)$, where $D_0$ is an ICC nonamenable subgroup of a torsion-free hyperbolic group, and $F=\mathbb F_4$ is the free group on $4$ generators.
We let $\varepsilon :V_0\to D_0$ be the corresponding quotient map whose kernel is the internal direct sum $\bigoplus_{d\in D_0}(F^{(\N )})_d$. 
The summands satisfy $v(F^{(\N )})_d v^{-1} = (F^{(\N )})_{\varepsilon (v)d}$ for each $v\in V_0$ and $d\in D_0$.
We identify $F^{(\N )}$ with the subgroup $(F^{(\N )})_e$ of $V_0$.   

We also fix, by Proposition \ref{torsion-free-family}, a countably infinite family $(A_i)_{i\in\N}$ of $4$-generated ICC torsion-free groups in the class $\mathcal{C}$ which are pairwise not measure equivalent; we may assume $A_i=F/N_i$ for some normal subgroup $N_i$ of $F$.

For $x\in \{ 0,1\}^\N$ we let $I_x\coloneqq \{ i\in \N : x(i)=1 \}$ and define the normal subgroup $K_x$ of $F^{(\N )}$ by $K_x\coloneqq (\bigoplus_{i\in I_x}N_i)\oplus F^{(\N \setminus I_x )}$.
That is, $K_x$ consists of all elements of $F^{(\N )}$ whose $i$-th coordinate belongs to $N_i$ for every $i\in I_x$.
Then $F^{(\N )}/K_x \cong \bigoplus_{i\in I_x}A_i$.
Moreover, the map $x\mapsto K_x$ is continuous.

Next, we recall some details from the proof of \cite[Lemma 2.12]{CIOS23a}. 
Each normal subgroup $K$ of $F^{(\N )}$ is also normal in $\bigoplus_{d\in D_0}(F^{(\N )})_d$ and hence a given conjugate $vKv^{-1}$ of $K$ by $v\in V_0$ only depends on $\varepsilon (v)$; we denote this conjugate by $(K)_d$ when $\varepsilon (v)=d$. 
The normal closure of $K$ in $V_0$ is then the internal direct sum $\ll K \rr = \bigoplus_{d\in D_0}(K)_d$. 
From this, one sees that $V_0/\ll K\rr \in \mathcal{WR}(F^{(\N )}/K , D_0)$. 
This also makes clear that the map $\mathrm{NSub}(F^{(\N )})\to \mathrm{NSub}(V_0)$, $K\mapsto \ll K\rr$, is continuous.

Thus, taking $G_x\coloneqq V_0/\ll K_x\rr$, the family $(G_x)_{x\in \{ 0, 1\}^\N}$ is Borel, and $G_x \in \mathcal{WR}(\bigoplus_{i\in I_x}A_i,D_0)$. As in the proof of Theorem \ref{A}, these groups are pairwise not measure equivalent, and each $G_x$ is a torsion-free ICC group with property (T).

Since each $G_x$ has property (T), we have $\beta_1^{(2)}(G_x)=0$, hence \cite[Theorem 1.5]{AG} implies that the groups in the family $(G_x\ast \mathbb{Z} )_{x\in \{ 0,1\}^{\N}}$ are pairwise not measure equivalent.
\end{proof}

\subsection{Proof of Theorem \ref{ctblered}}

\begin{proof}[Proof of Theorem \ref{ctblered}]
We fix $n\in \N$ such that the group $V_0$ from Corollary \ref{Borelfam} is $(n-1)$-generated, and hence all the groups in the Borel family $(G_x\ast \mathbb{Z} )$ produced by Corollary \ref{Borelfam} are $n$-generated.
We work in the setting of the proof of Proposition \ref{WR}, using this value of $n$ as input. 
We recall that at the beginning of the proof of Proposition \ref{WR} we fix a nontrivial torsion-free finitely presented group $L$. 
The argument then produces: 
\begin{itemize} 
\item a rank $n$ free group $H$, 
\item a torsion-free hyperbolic group $B$ with a surjective homomorphism $\sigma : B\to L$ whose kernel $D\coloneqq\ker(\sigma)$ is ICC and nonamenable, and
\item a group $W\in \mathcal{WR}(H,B)$ with an associated short exact sequence
\[
\{ e\}\to \bigoplus_{b\in B}H^{(b)} \xrightarrow{\iota} W\xrightarrow{\rho} B\to \{ e\},
\]
such that the group $V\coloneqq \rho^{-1}(D)$ has property (T).
\end{itemize}

Here, the map $\iota$ is subgroup inclusion, $H^{(e)}$ is a copy of $H$ in $W$, and the groups $H^{(b)}$, $b\in B$, are conjugates of $H^{(e)}$ in $W$ generating the (internal to $W$) direct sum $H^{(B)}=\bigoplus_{b\in B}H^{(b)}$. 
These conjugates satisfy $wH^{(b)}w^{-1} = H^{(\rho (w)b)}$ for all $w \in W$ and $b\in B$. 
We identify $H$ with $H^{(e)}$, so that $H\leq W$.

Note that if $N$ is a normal subgroup of $H$, then $N$ is normal in $H^{(B)}$, so if $\rho (w_0)=\rho (w_1)=b$ then $w_0Nw_0^{-1}=w_1Nw_1^{-1}$, and we unambiguously denote this conjugate of $N$ by $N^{(b)}$.
Then $wN^{(b)}w^{-1}=N^{(\rho (w)b)}$ for all $w\in W$ and $b\in B$. 

Equip the space $Y\coloneqq\mathrm{NSub}(H)^L$, of all $L$-indexed sequences of normal subgroups of $H$, with the product topology, which is compact metrizable. 
We denote an element of $Y$ by $M$, and for $M\in Y$ we denote its coordinates by $M(\ell )$, for $\ell \in L$.
Let $L$ act on $Y$ by the left shift action given by $(\ell_0\cdot M)(\ell_1 )=M(\ell_0^{-1}\ell_1)$.
For $M\in Y$, we define the subgroup $K_{M}$ of $H^{(B)}$ by
\[
K_{M}\coloneqq \bigoplus_{b\in B}M(\sigma (b))^{(b)} .
\]
\begin{claim}\mbox{}
\begin{enumerate}
\item[(1)] $wK_{M}w^{-1}=K_{\sigma(\rho (w))\cdot M}$ for all $w\in W$ and $M\in Y$. 
\item[(2)] For each $M\in Y$ the group $K_{M}$ is normal in $V$, and $V/K_{M} \in \mathcal{WR}(\bigoplus_{\ell\in L}H/M(\ell ) , D)$.
\item[(3)] The map $Y\to\mathrm{NSub}(V)$, $M\mapsto K_M$, is continuous.
\end{enumerate}
\end{claim}  

\begin{proof}
(1): We have
\begin{align*}
wK_Mw^{-1} = \bigoplus_{b\in B}M(\sigma (b))^{(\rho (w) b)} 
&= \bigoplus_{b\in B} (\sigma (\rho (w))\cdot M)(\sigma (\rho (w)b))^{(\rho (w)b)} \\
&= \bigoplus_{b\in B}(\sigma (\rho (w))\cdot M)(\sigma (b))^{(b)} \\
&= K_{\sigma (\rho (w))\cdot M}.
\end{align*} 

(2): Since $V=\mathrm{ker}(\sigma\circ\rho )$, it follows from part (1) that $K_M$ is normal in $V$.
Let $\tau :V\to V/K_M$ denote the quotient map. 
The restriction of $\rho$ to $V$ is then a composition $\rho |_V = \bar{\rho}\circ\tau$ where $\bar{\rho}: V/K_M \to D$ is surjective, with kernel $\tau (H^{(B)})$.  

Fix a section $s:L\to B$ of $\sigma$, and for each $d\in D$ define $P_d \coloneqq \tau \big(\bigoplus_{\ell\in L} H^{(ds(\ell ))} \big)$. 
Then
\[
P_d \cong \bigoplus_{\ell\in L}H^{(ds(\ell ))}/M(\ell )^{(ds(\ell ))} \cong \bigoplus_{\ell\in L}H/M(\ell ).
\]
The group $\ker (\bar{\rho})$ is the internal direct sum $\bigoplus_{d\in D}P_d$, and for $v\in V$ we have
\[
\tau (v)P_d\tau (v)^{-1} =\tau \big(\bigoplus_{\ell\in L} H^{(\bar{\rho}(\tau (v))ds(\ell ))} \big) = P_{\bar{\rho}(\tau (v))d}.  
\]
This shows $V/K_M \in \mathcal{WR}(\bigoplus_{\ell\in L}H/M(\ell ),D )$.

(3): This map takes values in $\mathrm{NSub}(H^{(B)})$, so it is enough to show that, for each $g\in H^{(B)}$, the preimage of the set $\{ K\in\mathrm{NSub}(H^{(B)}) : g\in K \}$ is clopen in $Y$. 
Given $g\in H^{(B)}$ there exist unique $b_0,\dots ,b_{r-1}\in B$ and $k_0\in H^{(b_0)},\dots , k_{r-1}\in H^{(b_{r-1})}$ with $g=k_0\cdots k_{r-1}$. 
Choose $w_0,\dots ,w_{r-1}\in W$ with $\rho(w_j)=b_j$ for each $0\leq j<r$. 
Then 
\[
\{ M\in Y :g\in K_M \} = \bigcap_{0\leq j<r}\{ M\in Y : w_j^{-1}k_jw_j\in M(\sigma(b_j)) \} ,
\]
which is clopen.
\end{proof}

Thus, for $M\in Y$, conjugation by $w\in W$ induces an isomorphism $V/K_M \cong V/K_{\sigma (\rho(w))\cdot M}$. 
The continuous map $Y\to\mathrm{Quot}(V)$, $M\mapsto V/K_M$, therefore takes points in the same $L$-orbit to isomorphic quotients of $V$.

Consider now a Borel action of $L$ on an uncountable standard Borel space $X$, and let $\mathcal{R}(L\curvearrowright X)$ denote the orbit equivalence relation generated by this action.

By Corollary \ref{Borelfam} and our choice of $n$, we may find a Borel map $X\to \mathrm{NSub}(H)$, $x\mapsto N_x$ such that the groups $(H/N_x )_{x\in X}$ are ICC torsion-free groups in the class $\mathcal{C}$ which are pairwise not measure equivalent.
We then obtain an $L$-equivariant map $X\to Y=\mathrm{NSub(H)}^L$, $x\mapsto M_x$ given by: $M_x(\ell ):=N_{\ell^{-1}\cdot x}$.

Thus, the Borel map $\Phi : X\to\mathrm{Quot}(V)$, given by $\Phi (x)\coloneqq V/K_{M_x}$, sends points in the same $L$-orbit to isomorphic groups. 

\begin{claim}
The map $\Phi$ is a Borel reduction from $\mathcal{R}(L\curvearrowright X)$ to both isomorphism and measure equivalence on the space $\mathrm{Quot}(V)$. 
Every group in the image of $\Phi$ is ICC, torsion-free, has property (T), and has vanishing $\ell^2$-Betti numbers in all dimensions.
\end{claim}

\begin{proof}
For the Borel reduction statement, it remains to show that if $V/K_{M_x}$ is measure equivalent to $V/K_{M_y}$, then $x$ and $y$ are in the same $L$-orbit.
To see this, observe that since $V/K_{M_x} \in \mathcal{WR}(\bigoplus_{\ell\in L}H/M_x(\ell ) , D)$, and $V/K_{M_y} \in \mathcal{WR}(\bigoplus_{\ell\in L}H/M_y(\ell ) , D)$, and $D$ is an ICC nonamenable subgroup of a hyperbolic group, it follows from Theorem \ref{descend} that the groups $\big(\bigoplus_{\ell\in L}H/M_x(\ell )\big)^{(D)}$ and $\big( \bigoplus_{\ell\in L}H/M_y(\ell )\big) ^{(D)}$ are measure equivalent.
Then, by Theorem \ref{thm:MSinfiniteICCsum}, it follows that $H/M_x(e)$ is measure equivalent to $H/M_y(\ell )$ for some $\ell$, i.e., $H/N_x$ is measure equivalent to $H/N_{\ell ^{-1}\cdot y}$. 
Our choice of the map $z\mapsto N_z$ now ensures that $x=\ell^{-1}\cdot y$, as claimed.

Since $\Phi (x)= V/K_{M_x} \in \mathcal{WR}(\bigoplus_{\ell\in L}H/M_x(\ell ) , D)$, and since $D$ and each of the groups $H/M_x(\ell )=H/N_{\ell^{-1}\cdot x}$ are ICC and torsion-free, the group $\Phi (x)$ is ICC and torsion-free.
Each group $\Phi (x)$ has property (T) since $V$ does. 
All $\ell^2$-Betti numbers of $\Phi (x)$ vanish by the same argument as in the proof of Theorem \ref{A}, using the K\"{u}nneth formula and \cite[Theorem 7.2(6)]{Luck}. 
\end{proof}

Since $L$ was an arbitrary torsion-free finitely presented group, and $X$ was an arbitrary uncountable Borel $L$-space, we are free to take $L=\mathbb{F}_2$, and the action of $\mathbb{F}_2$ on $X$ to be one that generates a universal countable Borel equivalence relation, e.g., the shift action of $\mathbb{F}_2$ on $\{ 0, 1\}^{\mathbb{F}_2}$ \cite[Proposition 1.8]{DJK}.
Taking $Z\coloneqq\Phi (X)$, this completes the proof of Theorem \ref{ctblered}.
\end{proof}

\bibliographystyle{amsalpha}
\bibliography{biblio}

\end{document}